\documentclass{amsart}
\usepackage{amsmath, amsthm, amssymb}
\usepackage{verbatim}
\usepackage{tikz}
\usetikzlibrary{matrix}

\usepackage[mathscr]{eucal} 
\usepackage{mathrsfs} 
\usepackage{cancel}

\newcommand{\ii}{\sqrt{-1}}
\newcommand{\wt}[1]{\widetilde{#1}}

\newcommand{\wh}[1]{\widehat{#1}}

\newcommand{\veps}{\varepsilon}
\newcommand{\vphi}{\varphi}

\newcommand{\al}{\alpha}

\newcommand{\ga}{\gamma}
\newcommand{\de}{\delta}
\newcommand{\la}{\lambda}
\newcommand{\om}{\omega}

\newcommand{\sig}{\sigma}
\newcommand{\vka}{\varkappa}
\newcommand{\ka}{\kappa}

\newcommand{\Ga}{\Gamma}
\newcommand{\Om}{\Omega}

\newcommand{\bR}{\mathbb{R}}

\newcommand{\bC}{\mathbb{C}}

\newtheorem{thm}{Theorem}
\newtheorem{prop}[thm]{Proposition}
\newtheorem{lem}[thm]{Lemma}
\newtheorem{cor}[thm]{Corollary}

\theoremstyle{definition}
\newtheorem{defn}[thm]{Definition}
\newtheorem{remark}[thm]{Remark}

\numberwithin{thm}{section}
\numberwithin{equation}{section}
\renewcommand{\[}{\begin{equation}}
\renewcommand{\]}{\end{equation}}

\newcommand{\wed}{\wedge}

\newcommand{\ov}[1]{\overline{#1}}
\newcommand{\Bb}{{\bf B\,}}

\title[polar sets  on compact Hermitian manifolds]{Polar sets for $m$-subharmonic functions on compact Hermitian manifolds}
\author{S\l awomir Ko\l odziej and Ngoc Cuong Nguyen} 
\address{Faculty of Mathematics and Computer Science, Jagiellonian University, \L ojasiewicza 6, 30-348 Krak\'ow, Poland}
\email{slawomir.kolodziej@im.uj.edu.pl}
\address{Department of Mathematical Sciences, KAIST, 291 Daehak-ro, Yuseong-gu, Daejeon 34141, South Korea}
\email{cuongnn@kaist.ac.kr}

\date{\today}

\begin{document}

\maketitle

\begin{abstract} We prove a sharp decay of capacity of sublevel sets of a $(\om,m)$-subharmonic functions on a $n$-dimensional compact Hermitian manifold $(X,\om)$ which generalizes the case $m=n$  as well as the case $1\leq m\leq n$ on a compact K\"ahler manifold. We also obtain the full characterizations of polar sets of such functions in terms of the corresponding local and global capacities, and of the extremal functions.
\end{abstract}

\bigskip
\bigskip

\section{Introduction}
\label{sec:intro}

The introduction of a new capacity for plurisubharmonic functions by Bedford and Taylor  \cite{BT76, BT82} led to a positive answer to a question of Lelong \cite{Le50}: if   negligible sets are precisely pluripolar sets. They also used it to characterize pluripolar sets and to  simplify the proof of Josefson's theorem \cite{Jo} on the equivalence between locally and globally pluripolar sets. Subsequently, the first author  found an almost sharp uniform estimate for solutions of complex Monge-Amp\`ere equation whose the right hand side is well-dominated by the capacity \cite{Ko98, K05}.  The framework of pluripotential theory in \cite{BT76,BT82} has been generalized successfully to compact K\"aher manifolds by Guedj and Zeriahi \cite{GZ05}. Such a global pluripotential theory had a great impact in K\"ahler geometry as shown in the monograph \cite{GZ-book}.

In the real setting, $k$-convex functions are admissible solutions to real $k$-Hessian equations studied in \cite{CNS85}. The singularities of such functions have been studied thoroughly by Labutin in \cite{La} where the ideas of pluripotential theory  proved to be useful. We refer the reader to the survey of Wang \cite{Wa09}  and reference therein for more information on the equation and properties of this class of functions.

Later B\l ocki \cite{Bl05} initiated the study of potential theory for $m$-subharmonic ($m$-sh for short) functions   while smooth $m$-sh functions are admissible solutions to the complex Hessian equation which have been studied earlier in \cite{V88}, \cite{Li04},  independently. A major progress has been obtained by Dinew and the first author \cite{DK14, DK17} where the authors developed the weak solution theory with the right hand side in $L^p$ for the equation both in domains in $\bC^n$ and on  compact K\"ahler manifolds. This was a strong catalyst to push further the study of  singularities of $m$-sh functions and  Cegrell's approach \cite{Ce98} to complex Hessian equations which was done notably by Lu   \cite{chinh13b, chinh15}, Lu and Nguyen \cite{chinh-dong}.

Weak solutions of complex Hessian equations on Hermitian manifolds were studied in \cite{chinh13a} and \cite{KN3} after the works Tosatti and Weinkove \cite{TW10b} in the case of $m=n$ and Sz\'ekelyhidi \cite{Sz18} and  Zhang \cite{Zha17}  (independently) in the general case $1\leq m \leq n$. Since then, the topic has become attractive and there are  many recent works in this area. Let us mention just a few \cite{CM21}, \cite{CP22}, \cite{CX25}, \cite{D21}, \cite{DL21}, \cite{GN18}, \cite{GP22}, \cite{GPTW24} \cite{PT21},  \cite{GLu25}, \cite{LeN} and \cite{Su24}.

Now let $(X,\om)$ be a compact Hermitian manifold of dimension $n$ and let $m$ be an integer,  $1\leq m \leq n$. Denote by $SH_m(X, \om) $ the set of all $(\om,m)$-subharmonic functions on $X$.  We continue the study of potential theory for $(\om,m)$-subharmonic   functions 
on compact Hermitian manifolds initiated in  \cite{KN3, KN25}.  The crucial technical estimates carried out before for open sets
are here obtained in the compact setting. They allow to get more information on the singular (polar) sets of  $(\om,m)$-subharmonic functions, in particular the equivalence of notions of locally polar and globally polar sets and their characterization in terms of capacity.

For a Borel set $E\subset X$ the (global) $m$-capacity is given by
$$
	cap_m(E) = \sup \left\{ \int_E H_m(v) : v\in SH_m(X,\om), -1\leq v\leq 0 \right\}.
$$
Here the complex Hessian measure of a bounded $(\om,m)$-sh function $v$ is 
\[\label{eq:Hm}
	H_m (v) := (\om + dd^c v)^m \wed \om^{n-m}.
\]
We first show the following sharp estimate for the capacity of  of sublevel sets.

\begin{thm} \label{thm:intro-cap}
Let $v_0\in SH_m(X,\om)$ be such that $\sup_X v_0 =0$.  There exists a uniform constant $A$ depending only on $\om, m,n$ such that
$$
	cap_m(\{v_0 <- t\}) \leq \frac{A}{t} \quad \text{for every } t>0.
$$
In particular, $cap_m(P) =0$ if $P$ is a globally $m$-polar set.
\end{thm}

This generalizes a result  in \cite{DK12} which dealt with  the case $m=n$   and used the local argument. Here we use a global argument as in the case of compact K\"ahler manifolds \cite[Proposition~3.6]{GZ05} and \cite[Corollary~3.19]{chinh13b}, thus it provides also an alternative proof to the one in  \cite{DK12}. In fact, we prove a stronger uniform integrability of $(\om,m)$-sh function with respect to Hessian measures of bounded functions in this class (Theorem~\ref{thm:cap-int}). In the statement of the theorem we say that a set is  {\em globally} $m$-polar if there exists $u \in SH_m(X,\om)$ such that 
$$
	E \subset \{u = -\infty\}.
$$
A weaker result of this sort has been obtained recently by Fang \cite{Fa25a} where she considered a smaller $m$-capacity in which the supremum was taken over all  $v\in PSH(X,\om)$. This estimate is very useful in the  study of weak solutions to the complex Hessian equations \cite{KN3, Fa25b}.

Secondly, we give the characterizations of polar sets of $(\om,m)$-sh functions. Roughly speaking there are plenty of such globally $m$-polar sets.

\begin{thm}
\label{thm:intro-polar} Let $E\subset X$ be a Borel set. The following statements are equivalent. 
\begin{itemize}
\item[(a)] $E$ is a globally $m$-polar set.  
\item[(b)]  $E$ is a locally $m$-polar set.
\item[(c)] The relative extremal $m$-sh function $h_E^* \equiv 0$. 
 \item[(d)] $cap_m^*(E) = 0$.
 \item[(e)] The global extremal $m$-sh function $V_E^*\equiv +\infty$.
\end{itemize}
\end{thm}
Here the global extremal functions $h_E^*$ and $V_E^*$ are  $(\om, m)$-sh   analogues of the extremal functions in global pluripotential theory as  defined   by  Guedj and Zeriahi \cite{GZ05}. 
Thus, we get the same statements as for $\om$-plurisubharmonic ($\om$-psh) functions on compact K\"ahler manifolds \cite[Chapter 9]{GZ-book}. The equivalence between (a) and (b) is a version of Josefson's theorem for $(\om,m)$-subharmonic function on compact Hermitian manifolds. The case of quasi-psh functions on compact Hermitian manifolds has been proven by Vu \cite{Vu}  (see also \cite{GLu22}), while the case of $(\om, m)$-subharmonic functions on K\"ahler manifolds comes from   \cite{chinh-dong}. 

\bigskip

{\em Organization.} In Section~2 we briefly recall the basic definitions and properties of $(\om,m)$-sh functions. Then using the local definition in \cite{KN25} we define the complex Hessian operator for bounded functions on compact manifolds. Next, we state weak convergence theorems and the variants of Cauchy-Schwarz inequalities. Section~3 carries results from the  local setting in \cite{KN25} to the  global setting. Proposition~\ref{prop:cap-gen-est} is important for characterizations of polar sets.
Section~4 contains the proof of uniform integrability (Theorem~\ref{thm:cap-int}). The corresponding $m$-capacity for $(\om,m)$-sh functions is studied in Section~5. The proof of Theorem~\ref{thm:intro-cap} is derived and we also prove by a global argument for the quasi-continuity of $(\om,m)$-sh functions with respect to this capacity. Lastly, we provide the full characterizations of polar sets in Section~6.
 
\bigskip

{\em Acknowledgment.} 
The first  author was partially supported by Sheng
grant no. 2023/48/Q/ST1/00048 of the National Science Center, Poland.
The paper was written while the second author  visited the Center for Complex Geometry (Daejeon). He would like to thank Jun-Muk Hwang and Yongnam Lee for their kind support and exceptional hospitality. He is also grateful to the institution for providing perfect working conditions.  

\section{Hessian measures for bounded functions}
\label{sec:hess-measure}

Let $(X,\om)$ be a compact Hermitian manifold of dimension $n$ and $1\leq m \leq n$ be integer. 
Since $X$ is compact and $\om$ is a smooth Hermitian metric, there exists a constant $\Bb >0$ such  that 
\[\label{eq:curv-bound}
	- \Bb \om^2  \leq dd^c \om \leq \Bb \om^2, \quad
	-\Bb \om^3 \leq d\om \wed d^c \om \leq \Bb \om^3.
\]
This constant is fixed throughout the paper. We sometimes abuse this notation as we may need to multiply $\Bb$ by a multiple of $n$ and $m$ in estimates, still using the same letter  $\Bb$ for such constants.

\subsection{$(\om,m)$-subharmonic functions} \label{ssec:om-m}
Let $\om$ be a Hermitian metric on $\bC^n$ and let $\Om$ be a bounded open set in $\bC^n$. The positive cone $\Ga_m(\om)$, associated to $\om$, of real $(1,1)$-forms is defined as follows. A real $(1,1)$-form $\ga$ is said to belong to $\Ga_m(\om)$ if at any point $z\in \Om$,
$$
	\ga^k \wed \om^{n-k} (z)>0  \quad\text{for } k=1,...,m.
$$
Equivalently, in the normal coordinate system with respect to $\om$ at $z$, diagonalizing $\ga = \ii \sum_i \la_i dz_i\wed d\bar z_i$, we have $\la = (\la_1,...,\la_n) \in \Ga_m$, where 
$$	\Ga_m =\{ \la \in \bR^n : \sig_1(\la) >0, ...,\sig_m(\la)>0\}, 
$$
and  $\sig_k (\la) = \sum_{\la_{i_1<\cdots \la_{i_k}}} \la_{i_1} \cdots \la_{i_k}$ for $1\leq k \leq n$ is the $k$-elementary symmetric polynomial. 

If $u \in C^2(\Om, \bR)$ and $\om_u:= \om + dd^c u \in \ov{\Ga_m(\om)}$, then $u$ is called  an $(\om, m)$-subharmonic (sh) function in $\Om$. In general, an upper semicontinuous function $u : \Om \to [-\infty, +\infty)$ and $u\in L^1_{\rm loc}(\Om)$ is said to be $(\om,m)$-sh if it satisfies
\[\label{eq:om-m-sh-ineq}
	\om_u \wed \ga_1 \wed \cdots \wed \ga_{m-1} \wed \om^{n-m} \geq 0 \quad \text{for every } \ga_1, ...,\ga_{m-1} \in \Ga_m(\om)
\]
in the weak sense of currents. 
Let us denote $SH_m(\Om, \om)$ the set of all $(\om,m)$-sh functions in $\Om$. 

Moreover, we can consider a general positivity condition in \eqref{eq:om-m-sh-ineq} as follows. Let $\chi$ be a Hermitian $(1,1)$-form in $\bar\Om$.  If the function $u$ above satisfies
\[\label{eq:m-om-ineq}
	(\chi +dd^c u) \wed \ga_1 \wed \cdots \wed \ga_{m-1} \wed \om^{n-m} \geq 0 \quad \text{for every } \ga_1, ...,\ga_{m-1} \in \Ga_m(\om),
\]
instead of \eqref{eq:om-m-sh-ineq}, then it is called $m$-subharmonic with respect to $(\chi, \om)$. The space of all of such functions in $\Om$ is denoted by
\[\label{eq:chi-m-om}
	SH_m(\Om, \chi, \om).
\]
In a special case $\chi \equiv 0$, these functions are $m$-subharmonic  with respect to the metric $\om$,  or simply they are called $m-\om$-sh. 

On a compact Hermitian manifold $(X,\om)$ we use the following definition.

\begin{defn}[\cite{KN3}] An upper semi-continuous function $u: X\to [-\infty, +\infty)$ is called $(\om,m)$-sh if $u\in L^1(X)$ and $u\in SH_m(U, \om)$ for each coordinate patch $U\subset \subset X$.
\end{defn}

We  denote by $SH_m(X,\om)$ or simply by $SH_m(\om)$ (if there is no confusion),  the set of all $(\om,m)$-sh functions on $X$.

\subsection{Wedge product of bounded $(\om,m)$-subharmonic functions}

The complex Hessian measure for bounded $m-\om$-sh functions has been defined recently in \cite{KN25}. Using this we can define the wedge product of forms associated to bounded $(\om,m)$-sh functions on any small coordinate ball  $\Om \subset \subset X$ as follows. Since $\Om$ is 
is biholomorphic to a small ball in $\bC^n$, then we can find a strictly psh function $\rho$ in a neighborhood of $\bar\Om$ such that 
$$
	dd^c \rho \geq \om \quad \text{on } \bar\Om.
$$
Let $u\in SH_m(\om) \cap L^\infty(X)$. Then, $u +\rho$ is a bounded $m-\om$-sh function in $\Om$.
Hence, for $1\leq s \leq m-1$, the wedge product 
$$
 [dd^c(u+\rho)]^s  = dd^c (u+\rho) \wed \cdots \wed dd^c (u+\rho)$$
is defined inductively which results in a well-defined $(s,s)$-current of order zero \cite[Lemma~2.3]{KN25}. Moreover, 
$$
	 [dd^c(u+\rho)]^s \wed \om^{n-s} \quad \text{and}\quad  [dd^c(u+\rho)]^m \wed \om^{n-m}
$$
are positive Radon measures on $\Om$ by \cite[Theorem~3.3]{KN25}. By the choice of $\rho$, the smooth $(1,1)$-form $\tau := dd^c \rho -\om$ is positive.  The complex Hessian measure of $u$ in $\Om$ is given by
\[\label{eq:hess-measure}
	(\om + dd^c u)^m \wed \om^{n-m} := \sum_{s=0}^m (-1)^{m-s} \binom{m}{s} [dd^c (u+\rho)]^s \wed \tau^{m-s} \wed \om^{n-m}.
\]
If $u$ is smooth function, then this is an honest identity and therefore, the left hand side is a positive measure. If $u$ is only bounded, then we can take a sequence $\{u_j\}_{j\geq 1} \subset SH_m(\om) \cap C^\infty(X)$ such that $u_j \downarrow u$ on $X$ by \cite[Lemma~3.20]{KN3}.  The weak convergence theorem for decreasing sequences in [KN25, Lemma~5.1] allows us to define the Hessian measure as the well-defined positive measure on the right hand side.
By partition of unity we define $H_m(u)$ on the whole manifold $X$. The same construction can be applied for a tuple of $u_1,...,u_m \in SH_m(X,\om) \cap L^\infty(X)$.

We refer the readers to \cite{GN18}, \cite{GLu25} and \cite{KN3, KN25} for many more properties of general $m-\om$-subharmonic functions. The assumption "locally conformal K\"ahler" made in \cite{GN18} now is removed by the results on the wedge product forms associated to  bounded functions.

\subsection{Weak convergence} Since the weak convergence of measures is a local property, we can extend the results in \cite{KN25} to compact Hermitian manifolds. We state below two important convergence theorems for decreasing  and increasing sequences.  For simplicity we only state theirs simpler version for two sequences of functions, however, they are valid for the wedge products of  forms related to tuples of $k$ functions with $1\leq k\leq m$. We refer the readers to \cite[Section~5]{KN25} for those general statements. 

\begin{prop} \label{prop:conv-decreasing}
Let $\{v_j\}_{j\geq 1}, \{u_j\}_{j\geq 1} \subset SH_m(\om) \cap L^\infty(X)$ be sequences such that $v_j \downarrow v$ and $u_j \downarrow u$ in $X$ with $v, u \in SH_m(\om)\cap L^\infty(X)$. Then, $v_jH_m(u_j)$ converges weakly to $vH_m(u)$. 
\end{prop}

\begin{prop}\label{prop:conv-increasing} Let $\{v_j\}_{j\geq 1}, \{u_j\}_{j\geq 1}$ be  locally uniformly bounded sequences of $(\om,m)$-sh functions in $X$. Assume $v_j  \uparrow v$, $u_j \uparrow u$ with $u,v \in SH_m(\om)\cap L^\infty(X)$ almost everywhere as $j\to \infty$. Then, $v_jH_m(u_j)$ converges weakly to $vH_m(u)$. 
\end{prop}

\subsection{Weak comparison principles} 
Thanks to the quasi-continuity of $m-\om$-subharmonic functions in \cite[Theorem~4.9]{KN25} and equivalence between capacities (Lemma~\ref{lem:cap-equiv} below) we can remove the continuity assumption in  the weak comparison principle \cite[Theorem~3.7]{KN3}. 
Now the statement holds for bounded functions.

\begin{thm}[weak comparison principle] \label{thm:weak-cp}
Let $\vphi, \psi \in SH_m(\om)\cap L^\infty(X)$. Fix $0<\veps<1$ and use the following notations $$s_{\rm min}(\veps):= \inf_X [\vphi- (1-\veps)\psi] \quad\text{ and }\quad U(\veps, t) := \{\vphi<(1-\veps)\psi + s_{\rm min}(\veps)+t\}$$
for $s>0$. Then, for $0<t< \veps^3/16\Bb$,
\[
	\int_{U(\veps,t)} \om_{(1-\veps)\psi}^m \wed \om^{n-m}
	\leq (1+ \frac{C t}{\veps^m})\int_{U(\veps,t)} \om_\vphi^m \wed \om^{n-m} ,
\]
where $C>0$ is a uniform constant depending only on $n,m,\om$.
\end{thm}

Applying this comparison principle to the case $\psi =0$, $\vphi\in SH_m(\om)\cap L^\infty(X)$ for  a fixed $0<\veps<1$ and $t\in (0, \veps^3/16\Bb)$ we have 
\begin{cor}\label{cor:post-mass} If $\vphi \in SH_m(\om) \cap L^\infty(X)$, then
$
	\int_X H_m(\vphi) >0.
$
\end{cor}

An interesting consequence of the convergence theorem \cite[Lemma~5.1]{KN25} is  the following inequality. 

\begin{cor}\label{cor:max-prin} For  $\vphi, \psi \in SH_m(\om) \cap L^\infty(X)$,  
\[\label{eq:max-prin}
	H_m(\max\{\vphi, \psi\}) \geq {\bf 1}_{\{\vphi>\psi\}} H_m(\vphi) +{\bf 1}_{\{\vphi \leq \psi\}} H_m(\psi).
\]
Moreover, if $\vphi \leq \psi$, then 
\[\label{eq:max-prin-cor}
	{\bf 1}_{\{\vphi =\psi \}} H_m(\vphi) \leq {\bf 1}_{\{\vphi =\psi \}} H_m(\psi).
\]
\end{cor}

\begin{remark}
The first inequality was due to Demailly in pluripotential theory for psh functions and  it is often called  the maximum principle for $(\om,m)$-sh functions in \cite{GLu22, GLu25}. There is a different way to derive the above weak comparison principle via Corollary~\ref{cor:max-prin}. We refer the interested readers to Guedj and Lu \cite[Theorem~1.5]{GLu22} for a proof in the case $m=n$ which can be adapted easily to our case.
\end{remark}

Another useful result in the balayage procedure is as follows. 

\begin{prop}\label{prop:lift} Let $B \subset X$ be a small coordinate ball in $X$ and $\vphi \in SH_m(\om) \cap L^\infty(X)$. There exists $\wh\vphi \in SH_m(\om) \cap L^\infty(X)$ such that $\wh\vphi\geq \vphi$ and
$$
	\wh\vphi \equiv  \vphi \quad\text{in } X\setminus B, \quad
	H_m(\wh\vphi) =0 \text{ in } B.
$$
\end{prop}

\begin{proof}The proof is standard provided the solution to the homogeneous equation of $(\om,m)$-sh function in  small balls \cite[Theorem~3.15]{GN18}.
\end{proof}

\subsection{Cauchy-Schwarz's inequality}

Let $h$ be a smooth real-valued function and let  $\phi, \psi$ be Borel functions. We will need the following two versions of Cauchy-Schwarz's inequality  \cite[Lemma~2.3]{KN25} and \cite[Lemma~2.4]{KN25}   in this setting.  

The first one is  often applied  for the case of positive forms $T= \ga^{s} \wed \om^{n-m + \ell}$, where $\ga\in \Ga_m(\om)$ and  $0\leq s, \ell\leq m-1$ and $s+\ell = m-1$.

\begin{lem}\label{lem:CS-classic}  Let $T$ be a positive current of bidegree $(n-2,n-2)$. There exists a uniform constant $A$ depending on $\om,m,n$ such that
$$\begin{aligned}
&	\left |\int_X \phi\psi\; dh \wed d^c \om \wed T \right|^2 
&\leq 	A\int_X |\phi|^2 \;dh\wed d^c h \wed \om\wed T 
  \int_X |\psi|^2\; \om^2 \wed T. 
\end{aligned}$$
\end{lem}

On the other hand, the second one can be applicable for a $(n-2,n-2)$-form $\ga^{m-1}\wed \om^{n-m-1}$, where $\ga \in \Ga_m(\om)$, may not be positive in the integrand  of the left hand side. 
 \begin{lem} \label{lem:CS} 
There exists a uniform constant $A$ depending on $\om, n,m$ such for every $\ga \in \Ga_{m}(\om)$,
$$\begin{aligned}
&	\left |\int_X \phi\psi\; dh \wed d^c \om \wed \ga^{m-1}  \wed  \om^{n-m-1} \right|^2 \\
&\leq 	A \int_X |\phi|^2 \;dh\wed d^c h \wed \ga^{m-1} \wed \om^{n-m}  \times  \int_X |\psi|^2\;\ga^{m-1} \wed \om^{n-m+1}.
\end{aligned}$$
 \end{lem}

\subsection{Uniform constants and integral symbols} 
The uniform constants 
\[ \label{eq:const-short} A = A(\om,m, n) \text{ or } C = C(\om,m,n) \] appearing here and several times below are generic   they may be different from line to line. Moreover, in Section~\ref{sec:basic-int} and Section~\ref{sec:uni-int}  the integrals are always considered on the whole manifold $X$, so to simplify the notation we shall write
\[ \label{eq:int-short} \int f \eta:= \int_X f \eta \]
where  $f$ is a Borel function and $\eta$ is a smooth $(n,n)$-form.

\section{Basic integral estimates}
\label{sec:basic-int}

In this section we extend the integral estimates in \cite[Section~2.4]{KN25} to a compact Hermitian manifold. The basic idea is the same however the computations are slightly different. There is no boundary o a compact manifold so the integration by parts is easier without boundary terms. On the other hand, there will be extra terms as the differential operator $dd^c$ will act on more terms than the one in the local setting. Because of this we provide all details of  the proofs. Thanks to the convergence theorems in Section~\ref{sec:hess-measure} we may assume that all considered functions are smooth.

Let $-1 \leq v \leq u \leq 0$ be smooth $(\om,m)$-sh functions. Let $\phi$ be a smooth $(\om, m)$-sh function such that $-1\leq \phi \leq 0$. Denote
$$
	\om_\phi = \om + dd^c\phi, \quad h = u-v.
$$ 
We consider the integrals containing both potential $\phi$ and $v$, 
\[ \label{eq:eqks}
	e_{(q,k,s)} := \int h^{q+1} \om_\phi^k \wed \om_v^s \wed \om^{n-k-s},
\]
where $q\geq 0$, the integers $0\leq k\leq m$ and $0\leq s \leq m-k$. Notice that we are using the convention of uniform constants \eqref{eq:const-short} and  integral symbols \eqref{eq:int-short} in this section.

Our goal is to bound
$$
	e_{(q,m,0)} = \int h^{q+1} \om_\phi^m \wed \om^{n-m}
$$
by the integrals containing only potential $v$
$$
	e_{(r,0,i)} = \int h^{r+1} \om_v^i \wed \om^{n-i},
$$
where $i= 0,...,m$ and $0\leq r <q$. In other words, we will replace the potential $\phi$ by $v$. In the K\"ahler setting it is  relatively simply done via integration by parts as we do not have to deal with the torsion terms $dd^c \om$ and $d\om \wed d^c \om$. In the Hermitian setting  these terms  make the estimates  complicated. Following \cite[Section~2.4]{KN25}, we  use  variants of the  Cauchy-Schwarz inequality  to deal with the torsion terms appearing in integration by parts while replacing $\om_\phi$ by $\om_v$. 

The crucial estimates to deal with possibly non-positive forms come from \cite[Lemma~2.3]{KN3}. Namely, we have  for $1\leq k \leq m-1$, 
\[\label{eq:wed-prod-bound-b1}
	 dd^c (\om_\phi^k \wed \om^{n-k-1}) \leq  \Bb \sum_{\vka =0}^2 \om_\phi^{k-\vka} \wed \om^{n-k+\vka}  % \Bb [\om_\phi^k \wed \om^{n-k} + \om_\phi^{k-1} \wed \om^{n-k+1} + \om_\phi^{k-2} \wed \om^{n-k+2}].
\]
Moreover, for $0\leq k+s \leq m-1$,
\[\label{eq:wed-prod-bound-b2}
	dd^c [\om_\phi^k \wed \om_v^s \wed \om^{n-k-s-1}] \leq \Bb [\om_\phi+\om_v]^{k+s} \wed \om^{n-k-s}.
\]
Another useful inequality is as follows.
For $1\leq k \leq m$,
\[\label{rmk:equiv-forms-e}
	\sum_{i=0}^{k} \om_\phi^i \wed \om_v^{k-i} \wed \om^{n-k} \leq (\om_\phi +\om_v)^k \wed \om^{n-k} \leq C \sum_{i=0}^{k} \om_\phi^i \wed \om_v^{k-i} \wed \om^{n-k},
\]
where $C = C(\om,m,n)$ is a uniform constant. 
It follows that 
$$
	\sum_{i=0}^k e_{(q, i, k-i)} \leq \int h^{q+1} (\om_\phi +\om_v)^{k} \wed \om^{n-k} \leq C \sum_{i=0}^k e_{(q,i,k-i)}.
$$

We are ready to proceed with the bounds for $e_{(q,k,s)}$. As in \cite{KN25} we need to consider three cases as follows.

\begin{itemize}
\item Case 1: $k+s =m$,
\item Case 2: $k+s =m-1$,
\item Case 3: $k+s \leq m-2$.
\end{itemize}
The following lemma is the key technical tool which will be used repeatedly below.

\begin{lem} \label{lem:IBP-CS} Let $p\geq 1$ and $0\leq k \leq m-1$. There exists a constant $C = C(\om, m,n)$ such that
\begin{itemize}
\item
[(a)] for $0\leq s + k \leq m-1$,
\[ \begin{aligned}
& \int h^{p-1} dh \wed d^c h \wed \om_\phi^k \wed \om_v^s \wed \om^{n-k-s-1}   \\
& \leq  Ce_{(p-1, k,s+1)} + C \sum_{\vka=0}^2 \sum_{i=0}^{k+s-\vka} e_{(p, i, k+s-i-\vka)}.
\end{aligned}\]
\begin{comment}
where the last sum on the right hand side 
\[\begin{aligned}
& \leq \int h^p \om_\phi^k \wed \om_v^{s+1} \wed \om^{n-k-s-1} \\ 
&\quad + C \int h^{p+1} (\om_\phi + \om_v)^{k+s} \wed \om^{n-k-s} \\
&\quad + C \int h^{p+1} (\om_\phi + \om_v)^{k+s-1} \wed \om^{n-k-s+1} \\
&\quad + C \int h^{p+1} (\om_\phi + \om_v)^{k+s-2} \wed \om^{n-k-s+2}. \\
\end{aligned}\]
\end{comment}
Moreover, if $s=0$, we can take $e_{(p, i, k+s-i-\vka)} = e_{(p, i, 0)}$ for all $i$ in the sum.
\item
[(b)] for $0\leq s +k\leq m-3$,
\[\begin{aligned}
&	\int h^{p-1} dh \wed d^c h \wed \om_\phi^k \wed \om_v^s \wed \om^{n-k-s-1}  \\
&\leq e_{(p-1, k,s+1)} 
 + C \sum_{\vka=0}^1 \sum_{\vka'=0}^1 e_{(p, k-\vka,s-\vka') .} 
\end{aligned}
\]
\end{itemize}
\end{lem}

\begin{proof}  (a) Note first that $0\leq h \leq 1$, and $T := \om_\phi^k \wed \om_v^s \wed \om^{n-k-s-1}$, $\om_u \wed T$ are positive forms for $n-s-k-1 \geq n-m$. Therefore,
$$\begin{aligned}
	p(p+1) h^{p-1} dh\wed d^c h  \wed T
&=	 [dd^c h^{p+1} - (p+1) h^{p} dd^c h] \wed T \\
&\leq		 [dd^c h^{p+1} + (p+1) h^{p} \om_v] \wed T.
\end{aligned}$$  Hence, 
\[\label{eq:grad1}\begin{aligned}
&	\int h^{p-1} dh\wed d^c h \wed \om_\phi^k\wed \om_v^{s} \wed \om^{n-s-k-1} \\
&\leq \int (dd^c h^{p+1} + h^{p} \om_v) \wed \om_\phi^k\wed \om_v^{s}  \wed \om^{n-s-k-1}.
\end{aligned}\]
It remains  to estimate the product involving the first term in the bracket. By integration by parts and the basic inequality \eqref{eq:wed-prod-bound-b2},
\[\label{eq:grad2} \begin{aligned}
&	\int dd^c h^{p+1}  \wed \om_\phi^k\wed \om_v^{s} \wed \om^{n-k-s-1}\\
&	= \int h^{p+1} dd^c \left[ \om_\phi^k\wed\om_v^{s} \wed \om^{n-k-s-1}\right] \\
&\leq C \int h^{p+1} (\om_\phi+ \om_v)^{k+s} \wed \om^{n-k-s} \\
&\quad + C \int h^{p+1} (\om_\phi+ \om_v)^{k+s-1} \wed \om^{n-k-s+1} \\
&\quad + C \int h^{p+1} (\om_\phi+ \om_v)^{k+s-2} \wed \om^{n-k-s+2}, \\
\end{aligned}\]
where if $s=0$, then there is no $\om_v$ appearing on the right hand side as we can use the basic inequality \eqref{eq:wed-prod-bound-b1}.
Combining the last two inequalities the proof of the lemma follows.

(b) The proof is very similar but it is easier. We first have \eqref{eq:grad1}. Then, in the middle integral of  \eqref{eq:grad2}
one can express $$dd^c (\om^{n-s-k-1} \wed \om_\phi^k \wed \om_v^s) = \eta \wed \om^{n-m}
$$ for  smooth $(m-s-k,m-s-k)$-forms $\eta$ which are   the wedge products of $\om_\phi$,  $\om_v$, and the torsion terms either $dd^c\om$ or $d\om \wed d^c\om$. Since $\om_\phi, \om_v \in \Ga_m(\om)$ and the exponent in $\om$ is $n-m$, we can use the bounds \eqref{eq:curv-bound} for the torsion terms. Hence, 
$$\begin{aligned}
	\left|\int h^{p+1} \eta \wed \om_\phi^k \wed \om_v^s \wed \om^{n-m} \right| &\leq C \int h^{p+1} \om_\phi^k \wed \om_v^s \wed \om^{n-k-s} \\
&\quad +  C \int h^{p+1} \om_\phi^{k-1} \wed \om_v^s \wed \om^{n-k-s+1} \\
&\quad +  C \int h^{p+1} \om_\phi^{k} \wed \om_v^{s-1} \wed \om^{n-k-s+1} \\
&\quad +  C \int h^{p+1} \om_\phi^{k-1} \wed \om_v^{s-1} \wed \om^{n-k-s+2}.
\end{aligned}$$
The item (b) is proven.
\end{proof}

We are ready to begin with the simplest subcase of {\bf Case 1} when $s=0$. This is also  a starting point for the  induction argument. We are going to show that
\[\label{eq:start-ind}
	e_{(q,m,0)} \leq C e_{(q-1,m-1,1)} + Ce_{(q-1,m-1,0)}+ C\sum_{\vka=0}^2  e_{(q,m-2-\vka,0)}.
\]
Equivalently, 
\begin{lem} \label{lem:start-ind} Let $q\geq 2$ be integer. Then,
$$\begin{aligned}
	\int (u-v)^{q+1} \om_\phi^m \wed \om^{n-m} 
&	\leq C \int (u-v)^{q} \om_\phi^{m-1} \wed \om_v \wed \om^{n-m} \\
&\quad + C \int (u-v)^q \om_\phi^{m-1} \wed \om^{n-m+1} \\
&\quad + C \int (u-v)^{q+1} \om_\phi^{m-2} \wed \om^{n-m+2}\\
&\quad + C \int (u-v)^{q+1} \om_\phi^{m-3} \wed \om^{n-m+3}.
\end{aligned}$$
Here by convention $\om_\phi^k \wed \om^{n-k} \equiv \om^n$ for an integer $k\leq 0$.
\end{lem}

\begin{proof} 
Recall that $h := u-v \geq 0$. A direct computation gives
\[\label{eq:basic-computation}\begin{aligned}
	dd^c [h^{q+1} \om_\phi^{m-1} \wed \om^{n-m}] 
&= dd^c h^{q+1} \wed \om_\phi^{m-1}\wed \om^{n-m} \\
&\quad + d h^{q+1} \wed d^c (\om_\phi^{m-1}\wed \om^{n-m}) \\
&\quad - d^c h^{q+1} \wed d (\om_\phi^{m-1}\wed \om^{n-m}) \\
&\quad + h^{q+1} dd^c (\om_\phi^{m-1}\wed \om^{n-m})\\
& =: T_1 + T_2 + T_3 +T_4.
\end{aligned}\]
By integration by parts,
\[\label{eq:IBP-basic}\begin{aligned}
	\int h^{q+1} dd^c \phi \wed \om_\phi^{m-1} \wed \om^{n-m} 
&= \int \phi dd^c [h^{q+1} \om_\phi^{m-1}\wed \om^{n-m}]   \\
&= \int \phi (T_1 + T_2 + T_3 + T_4).
\end{aligned}\]

{\bf Case 1a: Estimate of $T_1$.} Compute
$$\begin{aligned}
 dd^c h^{q+1}  
=  q(q+1)  h^{q-1} dh \wed d^c h 
+ (q+1)  h^{q} [\om_u - \om_v]. 
\end{aligned}	
$$
Since $-1\leq \phi \leq 0$, $\om_u \wed \om_\phi^{m-1}\wed\om^{n-m}$ and $dh\wed d^ch\wed \om_\phi^{m-1} \wed\om^{n-m} \geq 0$, 
we derive
\[\label{eq:T1}
	\phi T_1 \leq  (q+1) h^q  \om_\phi^{m-1} \wed \om_v \wed\om^{n-m}.
\]
Then,
\[\label{eq:T1-int}
	\int \phi T_1 \leq (q+1) e_{(q-1, m-1,1)}.
\]

{\bf Case 1b: Estimate of $T_4$.} Using again the basic inequality \eqref{eq:wed-prod-bound-b1} we get
\[\label{eq:T4}
	dd^c (\om_\phi^{m-1} \wed \om^{n-m}) \leq C \sum_{\vka=0}^2 \om_\phi^{m-1-\vka} \wed \om^{n-m+1+\vka}. %[\om_\phi^{m-1} \wed \om^{n-m+1} + \om_\phi^{m-2} \wed \om^{n-m+2} + \om_\phi^{m-3}\wed \om^{n-m+3}].
\]
This implies that
\[\label{eq:T4-int}
	\int \phi T_4 \leq C [e_{(q,m-1,0)} + e_{(q,m-2,0)}+ e_{(q,m-3,0)}].
\]

{\bf Case 1c: Estimate of $T_2$ and $T_3$.} Since these two terms are bounded in the same way, we give details only for $T_2$. Compute
$$\begin{aligned}
	d h^{q+1} \wed d^c (\om_\phi^{m-1}\wed \om^{n-m}) 
&= 	(q+1) (m-1) h^q dh\wed d^c\om \wed \om_\phi^{m-2} \wed \om^{n-m} \\
&\quad + (q+1)(n-m) h^q dh\wed d^c\om \wed\om_\phi^{m-1} \wed \om^{n-m-1}.
\end{aligned}$$
Next we apply the Cauchy-Schwarz inequality in Lemma~\ref{lem:CS-classic} for the first term on the right hand side to obtain
\[ \label{eq:T23}\begin{aligned}
&	\left| \int \phi h^q dh\wed d^c \om \wed \om_\phi^{m-2} \wed \om^{n-m} \right|^2 \\
&\leq		C \int |\phi| h^{q-1} dh\wed d^c h \wed \om_\phi^{m-2} \wed \om^{n-m+1} \int |\phi| h^{q+1} \om_\phi^{m-2} \wed \om^{n-m+2} \\
 &\leq		C\left( \int  h^{q-1} dh\wed d^c h \wed \om_\phi^{m-2} \wed \om^{n-m+1} + \int  h^{q+1} \om_\phi^{m-2} \wed \om^{n-m+2} \right)^2 \\
 &\leq C\left[ e_{(q-1,m-2,1)} + e_{(q,m-2,0)} + e_{(q,m-3,0)} + e_{(q,m-4,0)} \right]^2,
\end{aligned}\]
where we used the fact $|\phi| \leq 1$ in the second inequality, and in the last inequality we invoked  Lemma~\ref{lem:IBP-CS}-(a) in the special case $(p, k,s) =(q,m-s,0)$. This yields that in the last three terms on the right hand side of that lemma only $\om_\phi$ appears.  Thus, 
\[\label{eq:T23-int}
	\left| \int \phi T_2 \right| \leq C  e_{(q-1, m-2, 1)} + C \sum_{\vka=0}^2 e_{(q, m-2-\vka,0)}.
\]
This completed the estimate of $T_2$ and $T_3$.

From the estimates in \eqref{eq:T1}, \eqref{eq:T4-int} and \eqref{eq:T23-int} for  $T_1, T_4$, $T_{2}$ and $T_3$ the proof follows.
\end{proof}

The general inequality in the Case 1 is stated as follows.
\begin{lem}\label{lem:case1} For $1\leq k\leq m$ and $k+s= m$ and $q\geq 2$,
$$\begin{aligned}
	e_{(q,k,s)} 
&\leq c_k \sum_{i=0}^{k-1} e_{(q-1,i, m-i)} + C \sum_{i=0}^{m-1} e_{(q-1-m+k,i,m-1-i)} \\
&\qquad+ C\sum_{\vka=0}^2\sum_{i=0}^{m-2-\vka} e_{(q,i, m-2-i-\vka)}.
\end{aligned}$$
\end{lem}

\begin{proof} We prove by  induction in decreasing  $k$ starting with $k=m$. For $k=m$, it is the content of Lemma~\ref{lem:start-ind} (thus if $m=1$ we are done). Assume that it is true for every $k+1 \leq \ell \leq m$, i.e., we have
\[\label{eq:ind-hyp} \begin{aligned}
	e_{(q, \ell, m-\ell)} 
&\leq c_\ell' \sum_{i=0}^{\ell-1} e_{(q-1,i, m-i)} + c_\ell' \sum_{i=0}^{m-1} e_{(q-1-m+\ell, i, m-1-i)} \\ 
&\qquad+ C\sum_{\vka=0}^2\sum_{i=0}^{m-2-\vka} e_{(q,i, m-2-i-\vka)}.
\end{aligned}
\] This implies that for $k+1\leq \ell \leq m$,
\[\label{eq:ind-hyp-cq}\begin{aligned}	
e_{(q, \ell, m-\ell)}  
&\leq c_\ell \sum_{i=0}^{k} e_{(q+k-\ell,i,m-i)} + c_\ell \sum_{x=0}^{\ell -k} \sum_{i=0}^{m-1} e_{(q-1-m+\ell -x,i,m-1-i)} \\
&\qquad +  C\sum_{\vka=0}^2\sum_{i=0}^{m-2-\vka} e_{(q,i, m-2-i-\vka)} \\
&\leq c_\ell \sum_{i=0}^{k} e_{(q-a,i,m-i)} + c_\ell \sum_{i=0}^{m-1} e_{(q-a,i,m-1-i)} \\
&\qquad +  C\sum_{\vka=0}^2\sum_{i=0}^{m-2-\vka} e_{(q,i, m-2-i-\vka)}, \\
\end{aligned}\]
where we set
\[ \label{eq:new-a}
a:= m-k.
\]

We need to prove the inequality for $\ell =k\geq 1$.
Denote 
$$
	\Ga = \om_\phi^{m-1} \wed \om^{n-m} \quad\text{and}\quad
	\Ga^{(s)} = \om_\phi^{m-1-s} \wed \om_v^s \wed \om^{n-m}.
$$
The strategy of the proof is the same as the one in Lemma~\ref{lem:start-ind} where it is done for $s=0$, i.e., $\Ga^{(0)} = \Ga$. The integrand of $e_{(q,k,s)}$ can be written as
$$
	h^{q+1}\om_\phi^{m-s} \wed \om_v^s \wed \om^{n-m} = h^{q+1}dd^c \phi \wed  \Ga^{(s)}  + h^{q+1} \om \wed \Ga^{(s)}.
$$
By integration by parts we have
$$
	\int h^{q+1} dd^c \phi \wed \Ga^{(s)} = \int \phi dd^c [h^{q+1} \Ga^{(s)}].
$$
Again a direct computation gives 
\[\label{eq:basic-cpt-ind}
\begin{aligned}
	dd^c [h^{q+1} \om_\phi^{m-1-s} \wed \om_v^s \wed \om^{n-m}] 
&= dd^c h^{q+1} \wed \om_\phi^{m-1-s}\wed \om_v^s\wed \om^{n-m} \\
&\quad + d h^{q+1} \wed d^c (\om_\phi^{m-1-s}\wed\om_v^s\wed \om^{n-m}) \\
&\quad - d^c h^{q+1} \wed d (\om_\phi^{m-1-s}\wed \om_v^s\wed \om^{n-m}) \\
&\quad + h^{q+1} dd^c (\om_\phi^{m-1-s}\wed\om_v^s\wed \om^{n-m})\\
& =: T_1 + T_2 + T_3 +T_4.
\end{aligned}\]

A similar consideration as in \eqref{eq:T1} gives
\[\label{eq:T1-x}
	\phi T_1 \leq (q+1) h^q \om_\phi^{m-1-s} \wed \om_v^{s+1} \wed \om^{n-m}
\]
and therefore, 
\[\label{eq:T1-int-x}
	\int \phi T_1 \leq (q+1) e_{(q-1,m-1-s, s+1)}.
\]
However, in the basic inequality \eqref{eq:wed-prod-bound-b2} the estimate for $T_4$ will have more terms when $s\geq 1$,
\[\label{eq:T4-x}
\begin{aligned}
	dd^c (\om_\phi^{m-1-s} \wed \om_v^s \wed \om^{n-m}) 
&\leq C (\om_\phi +\om_v)^{m-1} \wed\om^{n-m+1}  \\
&\quad + C (\om_\phi +\om_v)^{m-2} \wed \om^{n-m+2} \\
&\quad + C (\om_\phi + \om_v)^{m-3} \wed \om^{n-m+3}.
\end{aligned}\]
It follows from Remark~\ref{rmk:equiv-forms-e} that
\[\label{eq:T4-int-x}
	\int \phi T_4 \leq C \sum_{\vka=0}^2\sum_{i=0}^{m-1-\vka} e_{(q, i, m-1-i-\vka)}. %+   \sum_{i=0}^{m-2} e_{(q, i, m-2-i)} +  \sum_{i=0}^{m-3} e_{(q, i, m-3-i)} \right).
\]

Lastly, we deal with the new terms $T_2$ and $T_3$ compared with the ones in  Lemma~\ref{lem:start-ind}. Compute
$$\begin{aligned}
&d h^{q+1} \wed d^c (\om_\phi^{m-1-s} \wed \om_v^s \wed \om^{n-m}) \\
&=	(q+1)(m-1-s) h^q dh \wed d^c \om \wed \om_\phi^{m-s-2} \wed \om_v^s \wed \om^{n-m} \\
&\quad + (q+1)s h^q dh \wed d^c \om \wed \om_\phi^{m-s-1} \wed \om_v^{s-1} \wed \om^{n-m} \\
&\quad + (q+1)(n-m) h^q dh \wed d^c \om \wed \om_\phi^{m-s-1} \wed \om_v^s \wed \om^{n-m-1} \\
&=: T_{2a} + T_{2b} + T_{2c}.
\end{aligned}$$
We will see that the estimates for $T_{2a}$ and $T_{2b}$ are exactly the same and they are easier than the ones for $T_{2c}$. Because the exponent of $\om$ in these two is $n-m$ it allows us to use an easier Cauchy-Schwarz inequality (Lemma~\ref{lem:CS-classic}). Note that the sum of degrees of $\om_\phi$ and $\om_v$ is $m-2$. Hence, after using Lemma~\ref{lem:IBP-CS}-(a) this sum will be at most $m-1$. 

Now we give a detailed steps for estimation of $T_{2a}$ and $T_{2b}$. Using the Cauchy-Schwarz' inequality (Lemma~\ref{lem:CS-classic}) we get
\[\label{eq:T2a}\begin{aligned}
 &\left| \int \phi h^qdh \wed d^c \om \wed \om_\phi^{m-2-s} \wed \om_v^s \wed \om^{n-m} \right|^2 \\
 &\leq 	C \int  h^{q-1} dh \wed d^c h \wed \om_\phi^{m-2-s}\wed \om_v^s \wed \om^{n-m+1} \\
 &\qquad \times \int  h^{q+1} \om_\phi^{m-2-s} \wed \om_v^s \wed \om^{n-m+2} \\
 &\leq C \left( \int  h^{q-1} dh \wed d^c h \wed \om_\phi^{m-2-s}\wed \om_v^s \wed \om^{n-m+1}  + e_{(q,m-2-s, s)} \right)^2.
\end{aligned}\]
We now apply Lemma~\ref{lem:IBP-CS}-(a) with $(p,k,s)=(q, m-2-s,s)$ for the first integral in the bracket. Then,
\[\label{eq:T2a-1}\begin{aligned}
&\int h^{q-1} dh\wed d^c h \wed \om_\phi^{m-2-s}\wed \om_v^s \wed \om^{n-m+1}  \\
&\leq C e_{(q-1, m-2-s, s+1)}  + C \sum_{\vka=0}^2 \sum_{i=0}^{m-2-\vka} e_{(q, i, m-2-i-\vka)}.
\end{aligned}\]

Let us proceed with the harder estimate for $T_{2c}$. Recall from \eqref{eq:new-a} that 
$a:=m-k.$
The Cauchy-Schwarz inequality in Lemma~\ref{lem:CS} gives
$$
\begin{aligned}
&	I^2:= \left| \int \phi h^q \wed d^c\om \wed \om_\phi^{m-s-1} \wed \om_v^s \wed \om^{n-m-1}\right|^2 \\
&\leq C \int h^{q +1-a} (\om_\phi +\om_v)^{m-1} \wed \om^{n-m+1} \\
&\qquad \times \int h^{q-1+a} dh\wed d^c h\wed (\om_\phi +\om_v)^{m-1} \wed \om^{n-m}.
\end{aligned}
$$
By the standard Cauchy-Schwarz inequality  for a given $\veps>0$ (to be determined later) 
\[\label{eq:T2c}\begin{aligned}
&	I \leq \frac{C}{\veps} \int h^{q-a} (\om_\phi +\om_v)^{m-1} \wed \om^{n-m+1} \\
&\qquad + \veps \int h^{q + a} dh \wed d^ch \wed (\om_\phi +\om_v)^{m-1} \wed \om^{n-m+1}\\
&=: J_1 + J_2.
\end{aligned}\]
By using Remark~\ref{rmk:equiv-forms-e}, the first integral on the right hand side is bounded by
\[\label{eq:J1-bound}
	J_1\leq \frac{C}{\veps} \sum_{i=0}^{m-1} e_{(q-a-1,i, m-1-i)}.
\]
We continue to deal with the second integral $J_2$. We will use 
$$
	(\om_\phi +\om_v)^{m-1} \leq C_{m,n} \sum_{i=0}^{m-1} \om_\phi^i \wed \om_v^{m-1-i} \wed \om^{n-m}
$$
and then Lemma~\ref{lem:IBP-CS} for $(p,k,s) =(q+a, i, m-1-i)$. This gives a bound for the second integral by
$$\begin{aligned}
J_2 &\leq	C \veps \sum_{i=0}^{m-1} e_{(q+a, i, m-i)} + C\veps  \sum_{\vka=0}^2\sum_{i=0}^{m-1-\vka} e_{(q+a+1, i, m-1-i-\vka)}.  
\end{aligned}$$
Let us consider the first sum above which contains $e_{(q+a, k, m-k)}$. Write
$$\begin{aligned}
	\veps \sum_{i=0}^{m-1} e_{(q+a,i,m-i)}& =  \veps \sum_{i\geq k+1} e_{(q+a, i,m-i)} + \veps \sum_{i=0}^{k-1} e_{(q+a,i,m-i)} \\
&\quad	+ \veps e_{(q+a,k,m-k)}. 
\end{aligned}$$
Applying the induction hypothesis \eqref{eq:ind-hyp-cq}  to each term in the first sum on the right hand side we get 
$$\begin{aligned}
	\veps \sum_{i\geq k+1}^{m-1} e_{(q+a,i,m-i)} 
&\leq \veps b_k \left(  e_{(q,k,m-k)} + \sum_{i=0}^{k-1} e_{(q, i, m-i)}\right) \\
&+ \veps b_k \sum_{i=0}^{m-1} e_{(q,i,m-1-i)} +  C\sum_{\vka=0}^2\sum_{i=0}^{m-2-\vka} e_{(q+a,i, m-2-i-\vka)}, \\
\end{aligned}$$ 
where $b_k = \sum_{i=k+1}^{m-1} c_i$. Using the decreasing property of $e_{(p,k,s)}$ in $p$ for $e_{(q, \bullet, \bullet)}$ and $e_{(q+a, \bullet, \bullet)}$, we get
\[\label{eq:J2-bound}\begin{aligned}
	J_2 
&\leq \veps(1+ b_k) e_{(q, k, m-k)} + \veps (1+ b_k) \sum_{i=0}^{k-1} e_{(q-1, i, m-i)} \\
&\qquad + (\veps b_k+C\veps) \sum_{i=0}^{m-1} e_{(q-1-a,i,m-1-i)} +  C\sum_{\vka=0}^2\sum_{i=0}^{m-2-\vka} e_{(q,i, m-2-i-\vka)}.	
\end{aligned}\]
Notice that $q-1-a =q-1- m+k \geq 1$.  

Combining \eqref{eq:T2c} and the two bounds \eqref{eq:J1-bound} and \eqref{eq:J2-bound} of $J_1, J_2$ we have
\[\label{eq:T2c-bound}\begin{aligned}
	I 
&\leq \veps(1+ b_k) e_{(q, k, m-k)} + \veps (1+b_k) \sum_{i=0}^{k-1} e_{(q-1, i, m-i)} \\
&\qquad+ [\veps b_k + C/\veps + C\veps] \sum_{i=0}^{m-1} e_{(q-1-a, i, m-1-i)}  \\
&\qquad+ C\sum_{\vka=0}^2\sum_{i=0}^{m-2-\vka} e_{(q,i, m-2-i-\vka)}.	
\end{aligned}
\]

Combining the estimates \eqref{eq:T1-int-x}, \eqref{eq:T4-int-x},   \eqref{eq:T2a}, \eqref{eq:T2a-1} and  \eqref{eq:T2c-bound}  we derive
$$\begin{aligned}
	e_{(q,k,s)}
&\leq  C e_{(q-1,k-1,s+1)} + C \sum_{i=0}^{m-1} e_{(q, i, m-1-i)} \\
&\quad + \veps(1+b_k) \sum_{i=0}^{k-1} e_{(q-1, i, m-i)} \\
&\quad + \veps (1+b_k) e_{(q,k,s)} + \veps(1+b_k) \sum_{i=0}^{m-1} e_{(q-1-a,i,m-1-i)} \\
&\quad + C e_{(q-1,k-2,s+1)} 	+ C \sum_{\vka=0}^2 \sum_{i=0}^{m-2-\vka} e_{(q,i,m-2-i-\vka)}.
\end{aligned}$$
Now we can choose $\veps$ so that $\veps(1+b_k) =1/2$ and regroup the terms on the right hand side (decreasing the first parameter in $e_{(\bullet, \bullet, \bullet)}$ if necessary). This implies for a possibly larger $C>0$ that
$$\begin{aligned}
	e_{(q, k,s)} 
&\leq (C + 1/2) \sum_{i=0}^{k-1} e_{(q-1, i, m-i)} + C \sum_{i=0}^{m-1}e_{(q-1-m+k, i, m-1-i)} \\
&\quad +  C \sum_{\vka=0}^2 \sum_{i=0}^{m-2-\vka} e_{(q,i,m-2-i-\vka)}.
\end{aligned}
$$ 
This proves the inequality \eqref{eq:ind-hyp} for $\ell =k$. Therefore, the proof of the corollary follows.
\end{proof}

Next, we consider {\bf Case 2}.

\begin{lem}\label{lem:case2} For $1\leq k \leq m-1$ and $k+s= m-1\geq 0$ and $q\geq 1$ we have
$$
	e_{(q,k,s)} \leq C e_{(q-1,m-2-s, s+1)}   + C \sum_{\vka=0}^2 \sum_{i=0}^{m-2-\vka} e_{(q,i, m-2-i-\vka)}.
$$
\end{lem}

\begin{proof} The basic computation is
\[\tag{\ref{eq:basic-cpt-ind}$'$}
\begin{aligned}
	dd^c [h^{q+1} \om_\phi^{m-2-s} \wed \om_v^s& \wed \om^{n-m+1}] \\
&= dd^c h^{q+1} \wed \om_\phi^{m-2-s}\wed \om_v^s\wed \om^{n-m+1} \\
&\quad + d h^{q+1} \wed d^c (\om_\phi^{m-2-s}\wed\om_v^s\wed \om^{n-m+1}) \\
&\quad - d^c h^{q+1} \wed d (\om_\phi^{m-2-s}\wed \om_v^s\wed \om^{n-m+1}) \\
&\quad + h^{q+1} dd^c (\om_\phi^{m-2-s}\wed\om_v^s\wed \om^{n-m+1})\\
& =: T_1' + T_2' + T_3' +T_4'.
\end{aligned}\]
Thus, the exponent of $\om$ in $T_{1}',...,T_{4}'$  increases by one.
The estimates of $T_1'$ and $T_4'$ are the same as the ones in \eqref{eq:T1-x} and \eqref{eq:T4-x}. Precisely,
\[\tag{\ref{eq:T1-x}$'$}
	\phi T_1' \leq (q+1) h^q \om_\phi^{m-2-s} \wed \om_v^{s+1} \wed \om^{n-m+1}
\]	
and therefore,
\[\tag{\ref{eq:T1-int-x}$'$}
	\int \phi T_1' \leq (q+1) e_{(q-1, m-2-s, s+1)}.
\]
The one for $T_4'$ is 
\[\tag{\ref{eq:T4-x}$'$}
\begin{aligned}
	dd^c [\om_\phi^{m-2-s} \wed \om_v^s \wed \om^{n-m+1}]
&\leq C (\om_\phi +\om_v)^{m-2} \wed\om^{n-m+1}  \\
&\quad + C (\om_\phi +\om_v)^{m-3} \wed \om^{n-m+2} \\
&\quad + C (\om_\phi + \om_v)^{m-4} \wed \om^{n-m+3}.
\end{aligned}\]
This implies
\[\tag{\ref{eq:T4-int-x}$'$}
	\int \phi T_4' \leq C \sum_{\vka=0}^2 \sum_{i=0}^{m-2-\vka}e_{(q, i, m-2-i-\vka)}.
\]

Next, the estimates for $T_2'$ and $T_3'$ are easier than for $T_2$ and $T_3$ above. Namely,
$$\begin{aligned}
&	d h^{q+1} \wed d^c (\om_\phi^{m-2-s} \wed \om_v^s \wed \om^{n-m+1}) \\
&=	(q+1)(m-2-s) h^q dh \wed d^c \om \wed \om_\phi^{m-3-s} \wed \om_v^s \wed \om^{n-m+1} \\
&\quad + (q+1)s h^q dh \wed d^c \om \wed \om_\phi^{m-2-s} \wed \om_v^{s-1} \wed \om^{n-m+1} \\
&\quad + (q+1)(n-m) h^q dh \wed d^c \om \wed \om_\phi^{m-2-s} \wed \om_v^s \wed \om^{n-m} \\
&=: T_{2a}' + T_{2b}' + T_{2c}'.
\end{aligned}
$$
We observe that the exponents of $\om$ in these three terms are at least $n-m$. It follows that the easier Cauchy-Schwarz inequality (Lemma~\ref{lem:CS-classic}) will be enough for all $T_{2a}', T_{2b}'$ and $T_{2c}'$. 

The estimation of $T_{2a}'$ and $T_{2b}'$ is as follows.
\[\tag{\ref{eq:T2a}$'$}\begin{aligned}
&\left|\int \phi h^q dh \wed d^c\om \wed \om_\phi^{m-3-s} \wed \om_v^s \wed \om^{n-m+1} \right|^2 \\
&\leq C \int h^{q-1} dh\wed d^c h \wed \om_\phi^{m-3-s} \wed \om_v^{s} \wed \om^{n-m+2} \\
&\qquad \times \int h^{q+1} \om_\phi^{m-3-s} \wed\om_v^s  \wed \om^{n-m+3} \\
&\leq C \left[ e_{(q-1,m-3-s,s+1)} + \sum_{\vka =0}^1 \sum_{\vka'=0}^1 e_{(q,m-3-s -\ka, s-\ka')}  \right]^2,
\end{aligned} 
\]
where we applied Lemma~\ref{lem:IBP-CS}-(b) for $(p,k,s) = (q+1, m-3-s,s)$ and $k+s \leq m-3$ with  the right hand side having less terms.

The estimation of $T_{2c}'$ is the one of $T_{2a}$ in Lemma~\ref{lem:case1}. In other words,
\[\tag{\ref{eq:T2a-1}$'$}
	\int \phi T_{2c}' \leq C e_{(q-1, m-2-s,s+1)} + C \sum_{\vka=0}^2 \sum_{i=0}^{m-2} e_{(q,i, m-2-i-\vka)}.
\]

We conclude the proof of lemma  from  (\ref{eq:T1-int-x}$'$), (\ref{eq:T4-int-x}$'$), (\ref{eq:T2a}$'$) and (\ref{eq:T2a-1}$'$).
\end{proof}

Lastly, we consider {\bf Case 3} which is the simplest one. 

\begin{lem} \label{lem:case3} For $1\leq k \leq m-2$ and $0\leq s \leq m-2-k$ and $q\geq 1$ we have
$$
	e_{(q,k,s)}  \leq C e_{(q-1,k-1,s+1)} + C\sum_{i=0}^{k+s-1}  e_{(q,i,k+s-i)}.
$$
\end{lem}

\begin{proof} For simplicity we consider the case $k+s =m-2$. The basic computation is
\[\tag{\ref{eq:basic-cpt-ind}$''$}
\begin{aligned}
	dd^c [h^{q+1} \om_\phi^{m-3-s} \wed \om_v^s& \wed \om^{n-m+2}] \\
&= dd^c h^{q+1} \wed \om_\phi^{m-3-s}\wed \om_v^s\wed \om^{n-m+2} \\
&\quad + d h^{q+1} \wed d^c (\om_\phi^{m-3-s}\wed\om_v^s\wed \om^{n-m+2}) \\
&\quad - d^c h^{q+1} \wed d (\om_\phi^{m-3-s}\wed \om_v^s\wed \om^{n-m+2}) \\
&\quad + h^{q+1} dd^c (\om_\phi^{m-3-s}\wed\om_v^s\wed \om^{n-m+2})\\
& =: T_1'' + T_2'' + T_3'' +T_4''.
\end{aligned}\]
A significant change here is that the exponent of $\om$ in $T_i''$ is at least $n-m$ and  Lemma~\ref{lem:IBP-CS}-(b) is also applicable for all these terms. 

The estimate for $T_1''$ is
\[\tag{\ref{eq:T1-int-x}$''$}
	\int \phi T_1'' \leq (q+1) e_{(q-1, m-3-s,s+1)}. 
\]
The estimate for $T_4''$ is
$$\begin{aligned}
	dd^c [\om_\phi^{m-3-s} \wed \om_v^s \wed \om^{n-m+2}]
&\leq C (\om_\phi +\om_v)^{m-3} \wed \om^{n-m+3}.
\end{aligned}$$
Hence,  
\[\tag{\ref{eq:T4-int-x}$''$}
	\int \phi T_4'' \leq C \sum_{i=0}^{m-3} e_{(q,i, m-3 -i)}. 
\]
The sum of indices $k+s$ on the right decreases by at least one.

The estimates for $T_2''$ and $T_3''$ are similar to $T_{2a}'$. Namely,
\[\tag{\ref{eq:T2a}$''$}
\begin{aligned}
&	\left|\int \phi d h^{q+1} \wed d^c (\om_\phi^{m-3-s} \wed \om_v^s\wed \om^{n-m+2}) \right| \\
&\leq C  [e_{(q, m-3-s,s+1)} + e_{(q, m-3-s, s)}+ e_{(q, m-4-s, s+1)}].
\end{aligned}\]

Combining (\ref{eq:T1-int-x}$''$), (\ref{eq:T4-int-x}$''$) and (\ref{eq:T2a}$''$) completes the proof of the lemma.
\end{proof}

Having the above results of Lemmas~\ref{lem:start-ind}, \ref{lem:case1}, \ref{lem:case2} and \ref{lem:case3} we can argue as in \cite[Proposition~2.15]{KN25} in local setting to get the main inequality. The only difference is that we may need to increase $q$ to be able to replace all $\phi$ on the left hand side.

\begin{prop} \label{prop:cap-gen-est} Let $e_{(q,k,s)}$ be the quantity defined in \eqref{eq:eqks}. Then, for $q \geq (n+1)m$,
$$
	e_{(q,m,0)} \leq C \sum_{s=0}^m e_{(0,0,s)},
$$
where $C = C(m,n,\om)$ is a uniform constant.
\end{prop}

\section{Uniform integrability}
\label{sec:uni-int}

In this section we prove the $L^1$-uniform integrability of normalized $(\om,m)$-sh functions with respect to Hessian measures of  bounded $(\om,m )$-sh functions.
Recall  the constant $\Bb >0$ defined in \eqref{eq:curv-bound} satisfies
$$
	- \Bb \om^2  \leq dd^c \om \leq \Bb \om^2, \quad
	-\Bb \om^3 \leq d\om \wed d^c \om \leq \Bb \om^3.
$$
Note again that we also use the integral symbol convention \eqref{eq:int-short} in this section.

\begin{thm} \label{thm:cap-int}Let $v_0 \in SH_m(X,\om) $ be such that $\sup_X v_0 =0$. Let $0\leq u \leq 1$ belong to $SH_m(X,\om)$. There exist uniform constants $C_{m}$ and $D_m$ depending only on $n,m,\Bb$ such that
$$
	\int_X -v_0 (\om + dd^c u)^m \wed \om^{n-m} \leq C_{m},
$$
and 
$$
\int_X -v_0 du \wed d^c u \wed (\om + dd^c u)^{m-1} \wed \om^{n-m} \leq D_{m}.
$$
\end{thm}

\begin{proof}
Thanks to the weak convergence theorems in Section~\ref{sec:hess-measure} we can assume that all considered functions are smooth. 
We prove the two bounds simultaneously by an induction argument.  Namely, we will prove that for $0\leq k \leq m$  the following two statements hold: there are uniform constants $C_\ell, D_\ell$ with $0\leq \ell \leq m$ depending only on $n, m$, $\Bb$ such that
\[ \tag{C$_k$}
\int -v_0 \om_u^\ell \wed \om^{n-\ell} \leq C_\ell \quad\text{for } 0\leq \ell \leq k, 
\]
and 
\[\tag{D$_{k}$}
\int -v_0 du \wed d^c u \wed \om_u^{\ell-1} \wed \om^{n-\ell} 	\leq D_{\ell} \quad \text{for }  0\leq \ell \leq k,
\]
where by convention (D$_{0}$) is the same as (D$_1$).
This is done as follows:
\begin{itemize}
\item Step 1:  (C$_0$)  and (D$_0$) are true,
\item Step 2:  (C$_{k}$) and (D$_{k-1}$) imply  (D$_{k}$) for $1\leq k \leq m$,
\item Step 3: (C$_k$) and (D$_{k}$) imply  (C$_{k+1}$) for $1\leq k+1\leq m$.
\end{itemize}
After proving these steps we get that both (C$_m$) and (D$_{m}$) hold. We will verify these steps in \eqref{eq:mass-0}, Lemmas~\ref{lem:grad-0}, \ref{lem:step2} and \ref{lem:step3} below.
\end{proof}

Let us start with 
{\bf Step 1.}  The statement (C$_0$) holds true as we first have the basic bound
\[\label{eq:mass-0}
	\int -v_0 \om^n \leq C_0
\]
from \cite[Lemma~3.3]{KN3}, where  $C_0>0$ is a uniform constant.
Next, we verify the statement (D$_0$). Notice that the proof of this one contains  the main idea of induction arguments.

\begin{lem}\label{lem:grad-0} There exists a uniform constant $D_0$ depending only on $C_0$ and $\Bb$ such that
\[\label{eq:grad-0}
	J_0= \int -v_0 du \wed d^c u \wed \om^{n-1} \leq D_0.
\]
\end{lem}

\begin{proof}
Since $dd^c u^2 = 2 u dd^c u + 2 du \wed d^c u$, we have
$$\begin{aligned}
2J_0 
&= \int - v_0 dd^c u^2  \wed \om^{n-1} + \int v_0 u dd^c u\wed \om^{n-1} \\
&= \int - v_0 dd^c u^2 \wed \om^{n-1}  + \int v_0 u \om_u \wed \om^{n-1} 
 + \int -v_0 u  \om^{n}.
\end{aligned}$$
Since $v_0 u\leq 0$, it follows that
\[ \label{eq:J0-v1}
2J_0 \leq \int - v_0 dd^c u^2 \wed \om^{n-1} + C_{0} =: J'_0+ C_{0}.
\]
By integration by parts,
\[\label{eq:J0-v1-abc}\begin{aligned}
	J'_0 
&= \int -u^2 dd^c (v_0 \om^{n-1}) \\
&= \int -u^2 dd^c v_0  \wed \om^{n-1} + \int -u^2 v_0 dd^c ( \om^{n-1}) \\
&\quad + 2 \int -u^2 dv_0 \wed d^c (\om^{n-1}) \\
&=: J'_{0a} + J'_{0b} + 2 J'_{0c}.
\end{aligned}\]
Here, the factor 2 appeared in $J'_{0c}$ because $d v_0 \wed d^c \om^{n-1} = d \om^{n-1} \wed d^c v_0$.
It is easy to see that 
\[\label{eq:J0-v1-a}
	J'_{0a} = \int -u^2 \om_{v_0} \wed \om^{n-1} +  \int u^2  \om^{n} \leq \int u^2  \om^{n} \leq C_{0}.
\]
Since $ dd^c \om^{n-1} \leq \Bb \om^n$,  
\[\label{eq:J)-v1-b}
	J'_{0b} \leq \Bb C_{0}.
\]
It remains to bound $J'_{0c}$. Again, by integration by parts,
$$\begin{aligned}
J'_{0c} &= 2\int v_0 u du \wed d^c (\om^{n-1}) + \int v_0 u^2 dd^c (\om^{n-1}) \\
&\leq 2\int v_0 u du \wed d^c (\om^{n-1}) + \Bb C_0.
\end{aligned}$$
To deal with the remaining integral on the right hand side we use the Cauchy-Schwarz inequality (Lemma~\ref{lem:CS-classic})  and then the fact  that $0\leq u\leq 1$. This gives
$$\begin{aligned}
&	\left|\int v_0 u du \wed d^c (\om^{n-1})\right| \\
& \leq  \left(A n^2 \int -v_0 du \wed d^c u \wed \om^{n-1} \times \int -v_0 \om^n\right)^\frac{1}{2} \\
& \leq \frac{1}{4} \int -v_0 du\wed d^c u \wed \om^{n-1} + An^2 \int -v_0 \om^n.
\end{aligned}$$
Therefore, 
\[\label{eq:J0-v1-c} \begin{aligned}
	J'_{0c} 
&\leq \frac{1}{2} \int -v_0 du\wed d^c u \wed \om^{n-1} + (2A + \Bb) C_0  \\
&= \frac{1}{2} J_0 	 + (2n^2A+ \Bb) C_0 .
\end{aligned}\]

Combining \eqref{eq:J0-v1}, \eqref{eq:J0-v1-abc}, \eqref{eq:J0-v1-a}, \eqref{eq:J)-v1-b} and \eqref{eq:J0-v1-c} we get that
$$
	 J_0 \leq 2 (1 + n^2A + \Bb) C_0.
$$
This finished the proof of the lemma. 
\end{proof} 

 Next, we deal with {\bf Step 2.}
Let $1\leq k \leq m$. Assume that we have the  uniform bounds
\[\label{eq:mass-k}
	\int -v_0 \om_u^k \wed \om^{n-k} \leq C_\ell, \quad  0\leq \ell \leq k,
\]
and 
\[\label{eq:grad-k}
	\int -v_0 du \wed d^c u \wed \om_u^{\ell-2} \wed \om^{n-\ell+1} \leq D_{\ell-1}, \quad 1\leq \ell \leq k.
\]
Notice that by Step 1, we have these statements for  $k=0$. Moreover,  if \eqref{eq:mass-k} holds for $k=m$, then  (D$_m$)  is true by {\bf Step 2}, and therefore Theorem~\ref{thm:cap-int} will follow.

\begin{lem} \label{lem:step2} There exists a uniform constant $D_k$ depending on $\Bb$, $C_{\ell}$ and $D_\ell$ with $0\leq \ell \leq k-1$ such that
$$
	J: =\int -v_0 du \wed d^c u \wed \om_u^{k-1} \wed \om^{n-k} \leq D_{k}.
$$
\end{lem}
\begin{proof}
Since $dd^c u^2 = 2 u dd^c u + 2 du \wed d^c u$, we have
$$\begin{aligned}
	2J 
&= \int - v_0 dd^c u^2 \wed \om_u^{k-1} \wed \om^{n-k} + \int v_0 u dd^c u\wed \om_u^{k-1} \wed \om^{n-k} \\
&= \int - v_0 dd^c u^2 \wed \om_u^{k-1} \wed \om^{n-k} + \int v_0 u \om_u^k \wed \om^{n-k} \\
&\quad + \int -v_0 u \om_u^{k-1} \wed \om^{n-k+1}.
\end{aligned}$$
Since $v_0\leq 0$ and $0\leq u\leq 1$, it follows  that
\[\label{eq:2j-step2}
2J \leq \int - v_0 dd^c u^2 \wed \om_u^{k-1} \wed \om^{n-k} + C_{k-1} =: J'+ C_{k-1}.
\]
By integration by parts,
\[\label{eq:j'-step2}\begin{aligned}
	J' 
&= \int -u^2 dd^c (v_0 \om_u^{k-1}\wed \om^{n-k}) \\
&= \int -u^2 dd^c v_0 \wed \om_u^{k-1} \wed \om^{n-k} + \int -u^2 v_0 dd^c (\om_u^{k-1}\wed \om^{n-k}) \\
&\quad + 2 \int -u^2 dv_0 \wed d^c (\om_u^{k-1} \wed \om^{n-k}) \\
&=: J'_1 + J'_2 + 2 J'_3.
\end{aligned}\]
It is easy to see that 
\[\label{eq:j1'-step2}
	J_1' \leq \int u^2 \om_u^{k-1} \wed \om^{n-k+1} \leq C_{k-1},
\]
and 
\[\label{eq:j2'-step2}
	J_2' \leq \Bb [C_{k-1}+ C_{k-2} + C_{k-3}].
\]
It remains to bound $J_3'$. By integration by parts,
\[\label{eq:j3'-step2}\begin{aligned}
J_3' &= \int v_0 u du \wed d^c (\om_u^{k-1} \wed \om^{n-k}) + \int v_0 u^2 dd^c (\om_u^{k-1} \wed \om^{n-k}) \\
&=: J_{3a}' + J_{3b}'.
\end{aligned}
\]
Clearly,
\[\label{eq:j3b'-step2}
J_{3b}' \leq \Bb[ C_{k-1} + C_{k-2} + C_{k-3}].
\]
Moreover, 
$$\begin{aligned}
	d^c (\om_u^{k-1} \wed \om^{n-k}) = (k-1) d^c\om \wed  \om_u^{k-2}\wed\om^{n-k} + (n-k) \om_u^{k-1} \wed \om^{n-k-1} \wed d^c \om.
\end{aligned}$$
Therefore,
\[\label{eq:j3a'-step2}\begin{aligned}
	J_{3a}' 
&= (k-1 )\int v_0 u du \wed d^c \om \wed \om_u^{k-2}  \wed \om^{n-k}   \\
&\quad + (n-k) \int v_0 u du \wed d^c\om \wed \om_u^{k-1} \wed \om^{n-k-1}\\
&=:J_{4}' + J_{5}'.
\end{aligned}\]
We can use the Cauchy-Schwarz inequality (Lemmas~\ref{lem:CS-classic}, \ref{lem:CS}) to derive bounds  for $J_4'$ and $J_5'$. Namely,  
since $k\leq m$, Lemma~\ref{lem:CS-classic} gives
$$\begin{aligned}
	|J_{4}'| 
&\leq \left( A \int -v_0 du \wed d^c u\wed \om_u^{k-2}\wed \om^{n-k+1}  \times \int -v_0 \om_u^{k-1} \wed \om^{n-k+1} \right)^\frac{1}{2} \\
&\leq D_{k-1} + A\, C_{k-1}.
\end{aligned}$$
On the other hand, we need to use Lemma~\ref{lem:CS} if $k=m$ to have 
$$\begin{aligned}
	|J_{5}'|  
&\leq \left( A \int -v_0 du \wed d^c u\wed \om_u^{k-1}\wed \om^{n-k}  \times \int -v_0 \om_u^{k-2} \wed \om^{n-k+2}\right)^\frac{1}{2} \\
&\leq \frac{1}{4} \int -v_0 du \wed d^c u\wed \om_u^{k-1}\wed \om^{n-k} + A\int -v_0 \om_u^{k-2} \wed \om^{n-k+2}  \\
&= \frac{1}{4} J +A\int -v_0 \om_u^{k-2} \wed \om^{n-k+2}.
\end{aligned}$$
Hence, 
$$	J_{3a}' \leq \frac{1}{4} J + D_{k-1} + A [C_{k-1} + C_{k-2}].
$$
Combining this with \eqref{eq:j3b'-step2} for $J_{3b}'$, we obtain
\[\label{eq:j3'-est-step2}
	J_{3}' \leq \frac{1}{4} J + D_{k-1} + A [C_{k-1}+ C_{k-2}] + \Bb[C_{k-1}+ C_{k-2}+ C_{k-3}].
\]

Finally, the proof of the lemma follows from \eqref{eq:2j-step2}, \eqref{eq:j'-step2}, \eqref{eq:j1'-step2},  \eqref{eq:j2'-step2}  and \eqref{eq:j3'-est-step2}.  We completed Step 2.
\end{proof}

Lastly, we verify {\bf Step 3.} Let $1\leq k+1\leq m$. Assume that both (C$_k$) and (D$_k$) hold. Then, we need to prove the following

\begin{lem}\label{lem:step3} There exists a uniform constant $C_{k+1}$ depending only on $\Bb$, $C_\ell$ and $D_\ell$ with $0\leq \ell \leq k$ such that 
$$
	I= \int -v_0  \om_u^{k+1} \wed \om^{n-k-1} \leq C_{k+1}.
$$
\end{lem}

\begin{proof}
Since 
$$\begin{aligned}
	I 
&= \int -v_0 \om_u^{k} \wed \om^{n-k} + \int -v_0 dd^c u \wed \om_u^k \wed \om^{n-k-1} \\
&\leq C_k + I',
\end{aligned}$$
where 
$$
	I': =\int -v_0 dd^c u \wed \om_u^k \wed \om^{n-k-1}.
$$
Hence, to bound $I$ it is enough to show that
\[
	I' \leq C_{k+1}.
\]
By integration by parts,
$$
	I' = \int -u dd^c [v_0 \om_u^k \wed \om^{n-k}].
$$
Compute
$$\begin{aligned}
	dd^c [v_0 \om_u^{k} \wed \om^{n-k-1}] 
&= dd^c v_0 \wed \om_u^k \wed \om^{n-k-1} \\
&\quad + v_0 dd^c [\om_u^{k} \wed \om^{n-k-1}]	 \\
&\quad + 2 d v_0 \wed d^c (\om_u^k \wed \om^{n-k-1}) \\
&:= e_1 +e_2 + 2a_1.
\end{aligned}$$
The bounds for the elementary terms $e_1, e_2$ are  easier. Namely,
$$\begin{aligned}
	\int -u e_1 
&= \int -u \om_{v_0} \wed \om_u^k \wed \om^{n-k-1} + \int u \om_u^k \wed \om^{n-k} \\
&\leq \int \om_u^k \wed \om^{n-k} \\
&\leq C_k.	
\end{aligned}$$
Similarly,
$$
	\int -u e_2  = \int -u v_0 dd^c [\om_u^k \wed \om^{n-k-1}] \leq \Bb[C_{k} + C_{k-1} + C_{k-2}].
$$
Now we consider the term $a_1$ requiring more advanced argument. By integration by parts
$$\begin{aligned}
	\int -u dv_0 \wed d^c (\om_u^k\wed \om^{n-k-1}) 
&= \int v_0 du \wed d^c (\om_u^k \wed \om^{n-k-1}) \\
&\quad + \int v_0 u dd^c (\om_u^k \wed \om^{n-k-1}) \\
&=: I_1' + I_2'.
\end{aligned}$$
Since $0\leq u\leq 1$ and 
$$ - dd^c (\om_u^k \wed \om^{n-k-1}) \leq \Bb [\om_{u}^k \wed \om^{n-k} + \om_u^{k-1} \wed \om^{k+1} + \om_u^{k-2} \wed \om^{n-k+2}] 
$$
we have 
$$
	I_2'=\int v_0 u dd^c (\om_u^k \wed \om^{n-k-1}) \leq \Bb[C_k + C_{k-1} + C_{k-2}].
$$
It remains to bound the first integral $I_1'$. Compute
\[ \label{eq:cpt-dc}
\begin{aligned}
d^c (\om_u^k \wed \om^{n-k-1}) 
&= k d^c \om \wed \om_u^{k-1} \wed \om^{n-k-1} \\ 
&\quad + (n-k-1) d^c \om \wed \om_u^{k} \wed  \om^{n-k-2}. \\
\end{aligned}\]
Applying the Cauchy-Schwarz inequality (Lemma~\ref{lem:CS}) we get
\[\begin{aligned}
&	\left|\int v_0 du \wed d^c \om \wed \om_u^{k-1} \wed \om^{n-k-1} \right|  \\
&\leq 	\left(A \int -v_0 du \wed d^c u \wed \om_u^{k-1} \wed \om^{n-k}  \times \int -v_0 \om_u^{k-1} \wed \om^{n-k+1} \right)^\frac{1}{2} \\
&\leq D_{k} + A\, C_{k-1}.
\end{aligned}\]
Thus, {\bf Step 3} is verified and the proof of Theorem~\ref{thm:cap-int} completed.
\end{proof}

\begin{remark} A weaker result concerning (C$_m$) has been obtained by Y. Fang \cite{Fa25a} where she assumed  $u$ to be $\om$-psh. 
\end{remark}

\section{capacity}
\label{sec:cap}

Recall that  for a Borel set $E\subset X$ the (global) $m$-capacity is given by
$$
	cap_m(E) = \sup \left\{ \int_E H_m(v) : v\in SH_m(X,\om), -1\leq v\leq 0 \right\}.
$$
A useful observation is that this capacity is comparable with similar quantity defined locally. In fact,  
let us consider a finite covering of $X$ by coordinate balls $\{B_i(s)\}_{i\in I}$ such that $B_i(2s)$ are still in holomorphic charts.  We  fix such a covering in what follows. For a Borel set $E\subset X$, we define another capacity
\[
	cap'_m(E) = \sum_{i\in I} c_m(E \cap B_i(s), B_i(2s)),
\]
where the local capacity  is given by 
\[\label{eq:local-cap}	c_m(E, \Om)  = \sup \left\{ \int_E (dd^c v)^m \wed \om^{n-m} : -1 \leq v \leq 0, \; v \text{ is $m-\om$-sh in $\Om$} \right\}.
\]
Notice that the class of $m-\om$-sh functions is obtained by applying the definition in \eqref{eq:m-om-ineq} for $\chi\equiv 0$.

\begin{lem}\label{lem:cap-equiv} The two capacities $cap_m$ and $cap'_m$ are equivalent. Namely, there exists a uniform constant $A_0$ depending only on $m,n,\om$ and the covering such that for every Borel set $E\subset X$,
\[
	\frac{1}{A_0} cap'_m(E) \leq cap_m(E) \leq A_0 cap'_m(E).
\]
\end{lem}

\begin{proof} The proof is identical to \cite[Lemma~3.5]{GN18} when we take $\chi = \al = \om$.
\end{proof}

Clearly we can see  from the definition that the capacity depends on the metric $\om$, however, it is a fixed metric. Furthermore, if $\al$ is another Hermitian metric, then we can consider 
\[\notag\label{eq:al-om-cap} \begin{aligned}
&	cap_{\al, m} (E) \\
&	= \sup \left\{\int_{E} (\al + dd^c v)^m \wed \om^{n-m}: v\in SH_m(X,\al, \om), -1\leq v\leq 0 \right\}.
\end{aligned}\]
This capacity is  also comparable with $cap_m'$. In other words,  by increasing $A_0>0$ (if necessary) we have
\[\label{eq:comparing-al-om-cap}
	\frac{1}{A_0} cap_{\al,m} (E) \leq cap_m(E) \leq A_0 cap_{\al,m}(E).
\]
In particular, if $cap_m(E) =0$ if and only if $cap_{\al,m}(E)=0$ for every Hermitian metric $\al$.

Theorem~\ref{thm:cap-int} allows us to obtain a sharp decay estimate for sublevel sets of $m$-subharmonic functions. This property is very useful and well-known for  $\om$-psh functions and $(\om,m)$-sh functions in the K\"ahler setting.

\begin{proof}[Proof of Theorem~\ref{thm:intro-cap}] The first inequality is an immediate consequence of Theorem~\ref{thm:cap-int} as
for a function $\phi \in SH_m(\om)$ and $0\leq \phi \leq 1$,
$$	
	\int_{\{v<-t\}} \om_\phi^m\wed \om^{n-m} \leq \int_X \frac{-v}{t} H_m(\phi) \leq \frac{A}{t}. 
$$
Taking supremum over all such functions $\phi$ we get the desired inequality. The second statement follows easily from the first one by letting $t\to \infty$.
\end{proof}

The following inequality is a direct consequence of \cite[Proposition~3.6]{KN3} as we enlarged the class of function to take supremum. This  in turn generalizes the ones due to Dinew and the first author \cite{DK14, DK17} in the K\"ahler setting.

\begin{lem} \label{lem:vol-cap} Let $1\leq q < n/m$. Then, there exists a uniform constant $A_q>0$ such that for every Borel set $E\subset X$,
$$
	V_{2n} (E) \leq A_q \;cap_m^q(E).
$$
\end{lem}

As a consequence we get a result which  has  been proven recently by Y. Fang \cite{Fa25a}.

\begin{cor} Let $1\leq q < n/m$. Let $v\in SH_m(\om)$ with $\sup_X v =0$. There exists a constant $A = A(q)>0$ such that
$$
	V_{2n} (v < -t) \leq \frac{A}{t^q} \quad \text{for } t>0.
$$
\end{cor}
\begin{proof} We have from the above lemma that
$$V_{2n} (v<-t) \leq cap_m(\{v<-t\})^q.$$ Then, the proof follows easily from Theorem~\ref{thm:intro-cap}.
\end{proof}

\begin{prop} \label{prop:cap-conv} Let $\{u_j\}_{j\geq 1} \subset SH_m(\om) \cap L^\infty(X)$ be such that $u_j$ decreasing to $u \in SH_m(\om)\cap L^\infty(X)$. Then, the sequence converges with respect to capacity, i.e., for each $\de>0$,
$$
 	\lim_{j\to \infty} cap_m(\{u_j - u>\de\}) =0.
$$
\end{prop}

\begin{proof} Let $-1 \leq \phi \leq 0$. Without loss of generality we may assume that $-1\leq u \leq u_j \leq 0$ by subtracting and dividing the large constants.   Then, we have from Markov's inequality and  Proposition~\ref{prop:cap-gen-est} that for $q = (n+1)m$,
$$\begin{aligned}
	\int_{\{u_j - u>\de\}} H_m(\phi) 
&\leq \frac{1}{\de^q}\int_X (u_j-u)^q H_m(\phi) \\
&\leq \frac{C}{\de^q} \sum_{s=0}^m \int_X (u_j-u) H_s(u),
\end{aligned}$$
where $H_s(u) = (\om + dd^c u)^s \wed \om^{n-s}$. Hence, taking supremum over such $\phi$ we derive 
$$
	cap_m(\{u_j-u>\de\})\leq \frac{C}{\de^q} \sum_{s=0}^m \int_X (u_j-u) H_s(u),
$$
The conclusion follows from  Lebesgue's dominated convergence theorem.
\end{proof}

We can state now one of the most basic properties of $(\om,m)$-sh functions. On compact K\"ahler manifolds it was proved earlier by Lu and Nguyen \cite{chinh-dong}. It generalizes the one for quasi-psh functions to a very general context.

\begin{prop}[quasi-continuity] Let $u \in SH_m(\om)$. For each positive number $\veps>0$ there exists an open set $U$ with $cap_m(U)<\veps$ such that $u$ restricted to $X \setminus U$ is continuous.
\end{prop}

\begin{proof} By subtracting a constant we may assume $u\leq 0$ on $X$.  It follows from Theorem~\ref{thm:intro-cap} that for $M>0$ large enough,
$$
	cap_m (\{u<-M\}) \leq \veps/2.
$$
Denote $v = \max\{u, -M\}$ and $U_0:= \{u<-M\}$. By \cite[Lemma~3.2]{KN3} there exists a sequence $\{v_j\}_{j\geq 1} \subset SH_m (\om) \cap C^\infty(X)$ that $v_j \downarrow v$. Proposition~\ref{prop:cap-conv} implies that  this sequence converges with respect to capacity. Thus, for each integer $k\geq 1$, there exists $j(k)$ such that the open set $U_{k} = \{v_{j(k)} > v +1/k\}$ satisfying
$$
	cap_m(U_k) \leq \veps/2^{k+1}.
$$
Then, $U =  U_0\cup \bigcup_{k\geq 1} U_k$ has capacity $cap_m(U) <\veps$ and  $v_{j(k)}$ converges uniformly to $\wt v =v$ on $X\setminus U$. Hence the restriction of $v$ to $X\setminus U$ is continuous.
\end{proof}

\begin{remark} The quasi-continuity can be obtained from the corresponding result in the local setting \cite[Theorem~4.9]{KN23} and Lemma~\ref{lem:cap-equiv}. However, the above proof could be useful if we considered the degenerate background metric as in \cite{GLu25}.
\end{remark}

\section{Characterization of polar sets}
\label{sec:polar}

In this section we will prove the characterization in Theorem~\ref{thm:intro-polar}. We define the $m$-polarity  locally as follows.

\begin{defn} A Borel set $E\subset X$ is called locally $m$-polar  if for each point $x\in E$, there exists a neighborhood $\Om \subset X$ and a $m$-sh function with respect to $\om$ (or $m-\om$-sh function) such that $u\neq -\infty$ and 
$ E \cap \Om \subset \{u = -\infty\}$.
\end{defn}

The space of $(\om,m)$-sh functions in a coordinate patch was defined in Section~\ref{ssec:om-m}. This space  was studied in more detail in \cite[Section~2, Section~9]{GN18}.
It follows from \cite[Proposition~7.7]{KN25} that 
a Borel set $E\subset X$ is a locally $m$-polar if and only if its outer local $m$-capacity (defined in \eqref{eq:local-cap})  $c_m^*(E \cap \Om,\Om) = 0$  on each coordinate patch $\Om \subset \subset X$.
Hence, we get immediately
\begin{lem} Let  $E\subset X$ be a subset. Then, $E$ is a locally $m$-polar set $\Leftrightarrow$
$cap_m'^*(E) =0 \Leftrightarrow cap_m^*(E) =0$.
\end{lem}
Here the outer capacity $cap_m^*$ is given by
$$
	cap_m^* (E) = \inf \{ cap_m (U) : E \subset U, \,U \text{ is open in } X \}
$$
and $cap_m'^*(E)$ is defined similarly.
Notice that $cap_m(E)$ is an inner regular, i.e.,
$$
	cap_m(E) = \sup \{ cap_m (K) : K \subset E, \, K \text{ is compact}\}.
$$

Assume $E$ is  globally $m$-polar and $u \in SH_m(X,\om)$ satisfies
\[\label{eq:global-polar}
	E \subset \{u = -\infty\}.
\]
If we take a strictly psh function $\rho$ in a local coordinate ball $B$ such that $dd^c \rho \geq \om$ in $B$, then $u+\rho$ is a $m-\om$-sh function in $B$ and $E \cap B \subset \{u+\rho = -\infty\}$. Hence, a globally $m$-polar set is  locally $m$-polar.  We will see later that the reverse inclusion is also true. A nice consequence is that there are plenty of unbounded $(\om, m)$-subharmonic functions on a compact Hermitian manifold. 

The global relative $m$-subharmonic extremal function is given by
$$
	h_E (z) = \sup\{ v(z) : v\in SH_m(X,\om), v\leq 0, v \leq -1 \text{ on } E \}.
$$
Here we write $h_{E}$ instead of $h_{m,E}$ as $m$ is already fixed. The function $h_E$  shares several properties with its counterpart in global pluripotential theory on compact K\"ahler manifolds.
The Choquet lemma shows  that there is  an increasing sequence of $v_j \in SH_m(\om)$, $-1\leq v_j \leq 0$ converging almost everywhere to the upper semi-continuous regularization $h_E^*$. Therefore,
$h_E^* \equiv 0$ if and only if there exists an increasing sequence of $(\om,m)$-sh function $\{v_j\}_{j\geq 1}$ such that 
\[\label{eq:seq-polar}
	v_j \leq 0,\quad v_j \leq -1 \text{ on } E, \quad \int_X |v_j| \om^n \leq \frac{1}{2^j}.
\]

By classical arguments in pluripotential theory \cite[Proposition~1.19]{K05} combined with \cite[Lemma~7.2]{KN25} and Proposition~\ref{prop:lift} we have

\begin{prop} \label{prop:relative-ext-basic} The following properties hold.
\begin{itemize}
\item[(a)] $h_E^* \in SH_m(X, \om)$ and $-1\leq h_E^*\leq 0$.
\item[(b)] $h_E^* =-1$ on $E \setminus P$ where $P$ is a globally $m$-polar set.
\item[(c)] Let $K_1 \supset K_2 \cdots $ be a sequence of compact sets in $\Om$ and $K= \cap_j K_j$. Then, $h_{K_j}^*$ increases almost everywhere to $h_K^*$. 
\item[(d)] If $h_{E_j}^* \equiv 0$ and $E = \cup_{j=1}^\infty E_j$, then $h_E^* \equiv 0$.
\item[(e)]  Let $E \subset X$. Then, $H_m(h_E^*) \equiv 0$ on the open set $\{h_E^*<0\}\setminus \ov{E}$. 
\end{itemize}
\end{prop}

The following results play the role of the capacity "formula" in the K\"ahler setting. This is the analogue of \cite[Lemma~7.5]{KN25} on compact manifolds.

\begin{lem} Let $E\subset X$ be a Borel set. 
\begin{itemize}
\item
[(a)] We have $$
	\int_X (-h_E^*) (\om + dd^c h_E^*)^m \wed \om^{n-m} \leq cap_m^*(E).
$$
\item
[(b)]  There exists a uniform constant $A = A(n,m,\om)$ such that
$$
	cap_m^* (E) \leq A \sum_{s=0}^m \int_X (-h_E^*) (\om + dd^c h_E^*)^s \wed \om^{n-s}.
$$
\end{itemize}
\end{lem}

\begin{proof} Let us prove the property (a). We have $-1\leq h_E^* \leq 0$ by Proposition~\ref{prop:relative-ext-basic}-(a). Assume first $E= \bar E$ is a compact set. It follows Proposition~\ref{prop:relative-ext-basic}-(e) that
$$
	\int_X (-h_E^*) H_m(h_E^*) \leq  \int_{E} H_m(h_E^*)  \leq cap_m(E).
$$
Next, assume $E = G$ is an open set. Let $\{K_j\}_{j\geq1}$ be an exhaustive sequence of compact sets which increases to $G$. As $G= \cup_{j\geq 1} K_j $ it is easy to see that $h^*_{K_j} \downarrow h_G = h_G^*$. The weak convergence theorem for decreasing sequences (Proposition~\ref{prop:conv-decreasing}) implies that $(-h_{K_j}^*) H_m(h_{K_j^*})$ converges weakly to $(-h_G) H_m(h_G)$ on $X$. Hence,
$$\begin{aligned}
	\int_X (-h_G) H_m(h_G) 
	&= \lim_{j\to \infty} \int_X (-h_{K_j}^*) H_m(h_{K_j^*}) \\ 
	&\leq \lim_{j\to\infty} cap_m(K_j) \\
	&= cap_m(G).
\end{aligned}$$

Finally, let $E$ be a general Borel set. By the definition of outer capacity, there exists a sequence of open sets $O_j \supset E$ such that $cap_m^* (E) = \lim_j cap_m(O_j)$. Replacing $O_j$ by $\cap_{1\leq s\leq j}O_s$ we may assume that $\{O_j\}$ is decreasing. Using Choquet's lemma we can find an increasing sequence of $(\om,m)$-sh functions $v_j \leq 0$ such that $v_j =-1$ on $E$ and $\lim_j v_j = h_E^*$ almost everywhere on $X$. Denote $G_j = O_j \cap \{v_j < -1 +1/j\}$. Then, $E\subset G_j \subset O_j$ and 
$$
	v_j -1/j \leq h_{G_j} \leq h_E.
$$
Therefore, $cap_m^*(E) = \lim_{j} cap_m(G_j)$ and $h_{G_j}$ increases to $h_E^*$ almost everywhere on $X$. Using the convergence theorem for increasing sequences (Proposition~\ref{prop:conv-increasing}) we get $(-h_{G_j}H_m(h_{G_j})) \to (- h_E^*) H_m(h_E^*)$ weakly on $X$ and thus
$$\begin{aligned}
	\int_X (-h_E^*) H_m(h_E^*) 
&= \lim_{j\to\infty} \int_X (-h_{G_j})H_m(h_{G_j}) \\
&\leq \lim_{j\to \infty} cap_m(G_j) \\
&= cap_m^*(E).
\end{aligned} $$ 
This completed the proof of (a).

The proof of (b) follows the  lines of the  one in \cite[Lemma~7.5]{KN25} provided  Proposition~\ref{prop:cap-gen-est} is  at our disposal.
\end{proof}

\begin{cor} \label{cor:relative-ext-polar} For $E\subset X$ a Borel set, $h_E^* \equiv 0$ if and only if $cap_m^*(E) =0$.
\end{cor}

\begin{proof}  If $h_E^* \equiv 0$, then $cap_m^*(E) =0$ by the second inequality of the above lemma. Conversely, assume $cap_m^*(E) =0$. % (then $h_E^* = 0$ a.e with respect to $H_m(h_E^*)$ by the first inequality of that lemma. The domination principle implies that $h_E^* =0$.)  
We shall prove $h_E^*\equiv 0$ by a contradiction argument. Assume $h_E^*$ is not identically equal to zero. Then $S_0:=\inf_{X} h_E^* <0$. We are going to apply the weak comparison principle (Theorem~\ref{thm:weak-cp}) for $\vphi = h_E^*$, $\psi =0$ and $\veps =1/2$, where $U(t) = \{h^*_E < S_0+ t\}$ is open and non-empty for $0< t < \min\{|S_0|/2, 1/32\Bb\}$. 
Then,
$$\begin{aligned}
	\int_{U(t)} \om^n 
& \leq C \int_{U(t)} (\om + dd^c h_E^*)^m \wed \om^{n-m} \\
& \leq   \frac{2C}{ |S_0|} \int_{X} -h_E^*(\om + dd^c h_E^*)^m \wed \om^{n-m} \\
&=0.
\end{aligned}$$
This leads to a contradiction as $U(t)$ is an open set whose volume is positive.
\end{proof}

We also consider the globally $m$-subharmonic extremal function modeled on the Siciak-Zaharjuta extremal function on compact manifolds, studied earlier in global pluripotential theory by Guedj and Zeriahi \cite[Theorem~9.17]{GZ-book}:
$$
	V_E(z) = \sup\{ v(z) : v\in SH_m(X,\om), v\leq 0 \text{ on } E \}.
$$
We have written $V_{E}$ instead of $V_{m,E}$ as $m$ is fixed to simplify the notations. The proof in the general case $1\leq m\leq n$ is very similar to the one for $m=n$ once the corresponding results  are supplied.

\begin{lem}\label{lem:global-ext-polar} Let $E \subset X$ be a Borel set. Let $V_E^*$ be the upper semi-continuous regularization of $V_E$.
\begin{itemize}
\item
[(a)] $E$ is a  globally $m$-polar set if and only if $\sup_X V_E^*=+\infty$, which is also equivalent to $V_E^* \equiv +\infty$.
\item
[(b)] If $E$ is not a globally $m$-polar set, then $V_E^* \in SH_m(\om) \cap L^\infty(X)$ and it satisfies $V_E^* \equiv 0$ in the interior of $E$, $(\om + dd^c V_E^*)^m \wed \om^{n-m} = 0$ in $X\setminus \bar E$ and
$$
	\int_{\bar E} H_m(V_E^*) = \int_X H_m(V_E^*).
$$
\end{itemize}
\end{lem}

\begin{proof} (a) We will show the following implications:
$$
	\text{globally $m$-polar } \Rightarrow V_E^* \equiv +\infty \Rightarrow \sup_X V_E^* =+\infty \Rightarrow \text{globally $m$-polar}.
$$
The middle one is obvious. Let us prove the first one. Assume $E \subset \{u =-\infty\}$ and $u\in SH_m(X,\om)$. So, $u + c \in SH_m(\om)$ and $u+c \leq 0$ on $E$ for every constant $c \in \bR$. This means $u + c \leq V_E$ and hence $V_E = +\infty$ in  $X \setminus \{u = -\infty\}$. Since the polar set has Lebesgue measure zero, we have $V_E^* = +\infty$ on $X$.  Next, we prove the last implication above. Assume now $\sup_X V_E^* = +\infty$. By Choquet's lemma there is a sequence $\{v_j\}_{j\geq 1} \subset SH_m(\om)$, $v_j = 0$ on $E$ increasing almost everywhere to $V_E^*$. By passing to a subsequence we may assume $\sup_X v_j \geq 2^j$. Denote
$$
	u_j = v_j - \sup_X v_j.
$$
It follows from \cite[Lemma~3.3]{KN3} and \cite[Lemma~9.12]{GN18} that  the sequence $\{u_j\}_{j\geq 1}$ is relatively compact in $L^1(X)$ and it satisfies that $\int_X |u_j| \om^n \leq C_0$ for a uniform constant $C_0$. Set $\vphi_\ell = \sum_{1\leq j\leq \ell} u_j/2^j.$ Then, the sequence $\vphi_\ell \in SH_m(X,\om)$ decreases to $\vphi \leq 0$ whose $L^1$-norm is  uniformly bounded. Thus, $\vphi \in SH_m(X,\om)$ and $\vphi(x) = -\infty$ for every $x\in E$. In other words, $E$ is a globally $m$-polar set. The proof of (a) is finished.

(b) Assume that $E$ is not a globally $m$-polar set. Then, $V_{E}^* \in SH_m(\om)$ and $\sup_X V_E^* = M_E <+\infty$. It is also clear that $V_E^* =0$ in the interior of $E$. The last conclusion follows from the lift property (Proposition~\ref{prop:lift}) and the standard balayage argument (see e.g \cite[Theorem~9.17]{GZ-book}).
\end{proof}

We are ready to show  the characterizations  in Theorem~\ref{thm:intro-polar}.

\begin{thm} Let $E\subset X$ be a Borel set. Then the following are equivalent. 
\begin{itemize}
\item[(a)] $E$ is a globally $m$-polar set.  
\item[(b)]  $E$ is a locally $m$-polar set.
\item[(c)] The relative extremal $m$-sh function $h_E^* \equiv 0$. 
 \item[(d)] $cap_m^*(E) = 0$.
 \item[(e)] The global extremal $m$-sh function $V_E^*\equiv +\infty$.
\end{itemize}
\end{thm}

\begin{proof} We already know from Lemma~\ref{lem:global-ext-polar}-(a), \eqref{eq:global-polar},  Lemma~\ref{lem:cap-equiv} and Corollary~\ref{cor:relative-ext-polar}  that
$$
	(e)\Leftrightarrow(a) \Rightarrow (b) \Leftrightarrow (d) \Leftrightarrow (c).
$$
Hence, it is enough to show that (d) $\Rightarrow$ (a). We will
use the observation \eqref{eq:seq-polar} for a sequence of Hermitian metrics. Namely,  put $\om_j = \om/2^j$ for integers $j\geq 1$. It is easy to see from \eqref{eq:comparing-al-om-cap} that $cap_{\om_j,m}^* (E) =0$. Applying  (c) $\Leftrightarrow$ (d) for $\om_j$ we  get $h_{\om_j, E}^* =0$, where
$$
	h_{\al,E}(z) = \sup\{v(z) : v\in SH_m(X, \al,\om) : v\leq 0, v\leq -1 \text{ on }E\}
$$
for another Hermitian metric $\al$ on $X$. 
Thus, there exists $v_j \in SH_m(X,\om_j,\om)$ (see the definition in \eqref{eq:chi-m-om}) such that 
$$v_j \leq 0,\quad v_j \leq -1 \text{ on } E, \quad \int_X |v_j| \om^n \leq \frac{1}{2^j}.
$$
Put $u_\ell = \sum_{j=1}^\ell v_j.$ Since $\sum_{j=1}^\ell \om_j \leq \om$, we have  $u_\ell \in SH_m(X,\om)$ for $\ell\geq 1$. Furthermore, this sequence is deceasing to the limit whose $L^1$-norm is uniformly bounded. Hence, $u_\ell \downarrow u\in SH_m(X,\om)$ and clearly $E \subset \{u = -\infty\}$.
\end{proof}

\begin{remark} 
 There is another proof of  (d) $\Rightarrow$ (e) which is enough to conclude the equivalences. One can  follow closely the strategy in \cite[Lemma~2.6]{GLu22} and \cite{Vu} supplying the needed ingredients proven above. 
\end{remark}

As  a consequence we have the counterpart of Proposition~\ref{prop:relative-ext-basic}-(c) for an increasing sequence of sets.

\begin{cor} Let $E_1\subset E_2\subset \cdots \subset X$ and $E = \cup E_j$. Then, $h_E^* = \lim_{j\to \infty} h_{E_j}^*$.
\end{cor}

\begin{proof} Provided the equivalence between locally and globally $m$-polar sets we can follow the  proof of \cite[Proposition~8.1]{BT82} or \cite[Corollary~4.7.8]{Kl91}.
\end{proof}

The next result generalizes the one for the Siciak-Zaharjuta extremal function in pluripotential theory (e.g., \cite[Section~5]{Kl91}).

\begin{cor}\mbox{}
\begin{itemize}
\item[(a)]
Let $E\subset X$ and $P$ be a $m$-polar set. Then, $V_{E\cup P}^* = V_{E}^*$.
\item
[(b)] If Let $E_1\subset E_2\subset \cdots \subset X$ and $E = \cup E_j$, then $V_E^* = \lim_{j\to \infty} V_{E_j}^*$.
\item
[(c)] Let $K_1 \supset K_2 \supset \cdots $ and $K = \cap_j K_j$. Then, $V_{K_j}$ increases to $V_K$ and hence $V_{K_j}^*$ increases a.e to $V_K^*$.
\item
[(d)] Let $E \subset X$ not be a $m$-polar set. Then, there exists a decreasing sequence of open subsets $G_j \supset E$ such that $V_E^* = \lim_{j\to \infty} V_{G_j}^*$.
\end{itemize}
\end{cor}

\begin{proof} The proof follows the lines of the one for  Proposition~9.19 in \cite{GZ-book} with obvious modifications in the current setting. 
\end{proof}


\begin{thebibliography}{000000000}

\bibitem[BT76]{BT76} E. Bedford and B. A. Taylor, {\it The Dirichlet problem for a complex Monge-Amp\`ere operator.} Invent. math. {\bf37} (1976)  1--44.

\bibitem[BT82]{BT82} E. Bedford and B. A. Taylor,  {\it A new capacity for plurisubharmonic functions.} Acta Math. {\bf149} (1982)  1--40.


\bibitem[Bl05]{Bl05} Z. B\l ocki, {\it Weak solutions to the complex Hessian equation,} Ann. Inst. Fourier (Grenoble) {\bf 55} (2005), no.~5, 1735--1756.

\bibitem[CNS85]{CNS85} L. Caffarelli, L. Nirenberg\ and\ J. Spruck, {\it The Dirichlet problem for nonlinear second-order elliptic equations. III. Functions of the eigenvalues of the Hessian,} Acta Math. {\bf 155} (1985), no.~3-4, 261--301.

\bibitem[Ce98]{Ce98} U. Cegrell,  {\it Pluricomplex energy,} Acta Math. {\bf 180:2} (1998) 187-217.

\bibitem[CX25]{CX25}
J. Cheng and Y. Xu, {\em Viscosity solution to complex Hessian equations on compact Hermitian manifolds}, J. Funct. Anal. {\bf 289} (2025), no.~5, Paper No. 110936, 52 pp.

\bibitem[CM21]{CM21} J. Chu, N. McCleerey, {\it Fully non-linear degenerate elliptic equations in complex geometry.} J. Funct. Anal. {\bf 281} (2021), no. 9, Paper No. 109176, 45 pp.

\bibitem[CP22]{CP22} T. Collins, S. Picard, {\it The Dirichlet problem for the $k$-Hessian equation on a complex manifold.}  Amer. J. Math. {\bf 144} (2022), no.6, 1641--1680.

\bibitem[DK12]{DK12}S. Dinew and S. Ko\l odziej, 
{\it  Pluripotential estimates on compact Hermitian manifolds.}   Adv. Lect. Math. (ALM), {\bf 21} (2012)  International Press, Boston. 

\bibitem[DK14]{DK14} S. Dinew\ and\ S. Ko\l odziej, {\it A priori estimates for complex Hessian equations,} Anal. PDE {\bf 7} (2014), no.~1, 227--244. 

\bibitem[DK17]{DK17}S. Dinew\ and\ S. Ko\l odziej, {\it Liouville and Calabi-Yau type theorems for complex Hessian equations.}  Amer. J. Math. {\bf 139} (2017), no.2, 403--415.


\bibitem[D21]{D21} W. Dong, {\it Second order estimates for a class of complex Hessian equations on Hermitian manifolds.}  
J. Funct. Anal. 281 (2021), no. {\bf 7}, Paper No. 109121, 25 pp.

\bibitem[DL21]{DL21} W. Dong, C. Li,  {\it Second order estimates for complex Hessian equations on Hermitian manifolds.} Discrete Contin. Dyn. Syst. {\bf 41} (2021), no. 6, 2619--2633.


\bibitem[Fa25a]{Fa25a} Y. Fang, {\em integrability of $(\om,m)$-subharmonic functions on compact Hermitian manifolds}, preprint 2025.

\bibitem[Fa25b]{Fa25b} Y. Fang, {\em Continuity of solutions to complex Hessian equations on compact Hermitian manifolds,} arXiv: arXiv:2510.14690.


\bibitem[GN18]{GN18} D. Gu, N.-C. Nguyen, {\it The Dirichlet problem for a complex Hessian equation on compact Hermitian manifolds with boundary.} Ann. Sc. Norm. Super. Pisa Cl. Sci. (5) {\bf 18 }(2018), no.4, 1189--1248.


\bibitem[GLu22]{GLu22}
V. Guedj and C.~H. Lu, Quasi-plurisubharmonic envelopes 2: Bounds on Monge-Amp\`ere volumes, Algebr. Geom. {\bf 9} (2022), no.~6, 688--713.


\bibitem[GLu25]{GLu25} V. Guedj and C.~H. Lu, {\it Degenerate complex Hessian equations on compact Hermitian manifolds}, Pure Appl. Math. Q. {\bf 21} (2025), no.~3, 1171--1194.

\bibitem[GZ05]{GZ05}
V. Guedj and A. Z\'eriahi, {\it Intrinsic capacities on compact K\"ahler manifolds}, J. Geom. Anal. {\bf 15} (2005), no.~4, 607--639.

\bibitem[GZ17]{GZ-book} V. Guedj\ and\ A. Zeriahi, {\it Degenerate complex Monge-Amp\`ere equations}, EMS Tracts in Mathematics, 26, European Mathematical Society (EMS), Z\"{u}rich, 2017.

\bibitem[GP24]{GP22} B. Guo, D. H. Phong, {\em On $L^\infty$ estimates for fully non-linear partial differential equations}, Ann. of Math. (2) {\bf 200} (2024), no.~1, 365--398.

\bibitem[GPTW24]{GPTW24} B. Guo et al., {\em On $L$$^\infty$ estimates for Monge-Amp\`ere and Hessian equations on nef classes}, Anal. PDE {\bf 17} (2024), no.~2, 749--756.

\bibitem[Jo78]{Jo} B. Josefson, {\em On the equivalence between locally polar and globally polar sets for plurisubharmonic functions on ${\bf C}\sp{n}$}, Ark. Mat. {\bf 16} (1978), no.~1, 109--11

\bibitem[Kl91]{Kl91} M. Klimek, {\it Pluripotential theory}, London Mathematical Society Monographs. New Series, 6, The Clarendon Press, Oxford University Press, New York, 1991.

\bibitem[Ko98]{Ko98}
S. Ko\l odziej, {\em The complex Monge-Amp\`ere equation}, Acta Math. {\bf 180} (1998), no.~1, 69--117;

\bibitem[Ko05]{K05} S. Ko\l odziej, {\it The complex Monge-Amp\`ere equation and pluripotential theory.} Memoirs Amer. Math. Soc. {\bf 178} (2005)  pp. 64.

\bibitem[KN16]{KN3}S. Ko\l odziej\ and\ N. C. Nguyen, {\it Weak solutions of complex Hessian equations on compact Hermitian manifolds}, Compos. Math. {\bf 152} (2016), no.~11, 2221--2248.

%\bibitem[KN23a]{KN22} S. Ko\l odziej\ and\ N.-C. Nguyen,{\em  The Dirichlet problem for the Monge--Amp\`ere equation on Hermitian manifolds with boundary,} Calc. Var. Partial Differential Equations {\bf 62} (2023), no.~1, Paper No. 1.

\bibitem[KN23b]{KN23} S. Ko\l odziej\ and\ N.-C. Nguyen, {\em  Weak solutions to Monge-Amp\`ere type equations on compact Hermitian manifold with boundary.} J. Geom. Anal. {\bf 33} (2023), no.1, Paper No. 15, 20 pp.

\bibitem[KN25]{KN25}S. Ko\l odziej\ and\ N.-C. Nguyen, {\em 
Complex Hessian measures with respect to a background Hermitian form,}
{\bf arXiv: 2308.10405.} To appear in APDE.

\bibitem[La02]{La}
D.~A. Labutin, {\em Potential estimates for a class of fully nonlinear elliptic equations}, Duke Math. J. {\bf 111} (2002), no.~1, 1--49.


\bibitem[LeN]{LeN}
L\^e~M\^au~Hai and V. Van~Quan, {\em Continuous solutions to complex Hessian equations on Hermitian manifolds}, J. Geom. Anal. {\bf 33} (2023), no.~12, Paper No. 368, 29 pp.

\bibitem[Le50]{Le50} P. Lelong, {\it Fonctions plurisousharmoniques et formes diff\'erentielles positives}, Gordon \& Breach, Paris-London-New York, 1968 distributed by Dunod \'Editeur, Paris, 1968


\bibitem[Li04]{Li04} S.-Y. Li, {\it On the Dirichlet problems for symmetric function equations of the eigenvalues of the complex Hessian,} Asian J. Math. {\bf 8} (2004), 87-106.

%\bibitem[LT94]{LT94} M. Lin\ and\ N. S. Trudinger, {\it On some inequalities for elementary symmetric functions,} Bull. Austral. Math. Soc. {\bf 50} (1994), no.~2, 317--326.

\bibitem[Lu13a]{chinh13a} H.-C. Lu, {\it Viscosity solutions to complex Hessian equations,} J. Funct. Anal. {\bf 264} (2013), no.~6, 1355--1379.

\bibitem[Lu13b]{chinh13b} H.-C. Lu, {\it Solutions to degenerate complex Hessian equations,} J. Math. Pures Appl. (9) {\bf 100} (2013), no.~6, 785--805. 

\bibitem[Lu15]{chinh15} C. H. Lu, {\it A variational approach to complex Hessian equations in $\Bbb{C}\sp n$,} J. Math. Anal. Appl. {\bf 431} (2015), no.~1, 228--259. 

\bibitem[LN15]{chinh-dong} H-C. Lu and V-D. Nguyen, {\it Degenerate complex Hessian equations on compact K\"ahler manifolds,}  Indiana Univ. Math. J. {\bf 64} (2015), no.~6, 1721--1745.

\bibitem[PT21]{PT21}
D.-H. Phong, T.-D. To,  {\it Fully non-linear parabolic equations on compact Hermitian manifolds.} Ann. Sci. \'Ec. Norm. Sup\'er. (4) {\bf 54} (2021), no. 3, 793-829.

\bibitem[Su24]{Su24}
W. Sun, {\em The weak solutions to complex Hessian equations}, Calc. Var. Partial Differential Equations {\bf 63} (2024), no.~2, Paper No. 57, 35 pp.

\bibitem[Sz18]{Sz18} G. Sz\'ekelyhidi, {\it Fully non-linear elliptic equations on compact Hermitian manifolds.} J. Differential Geom. {\bf 109} (2018), no.2, 337--378.


\bibitem[TW10]{TW10b} V. Tosatti and B. Weinkove, {\it The complex Monge-Amp\`ere equation on compact Hermitian manifolds.} J. Amer. Math. Soc. {\bf 23} (2010)  1187--1195.


\bibitem[V88]{V88} A. Vinacua, {\it Nonlinear elliptic equations and the complex Hessian,} Communications in partial differential equations 13.12 (1988), 1467-1497.

\bibitem[Vu19]{Vu} D.-V. Vu, {\em Locally pluripolar sets are pluripolar,} Internat. J. Math. {\bf 30} (2019), no.~13, 1950029, 13 pp


\bibitem[Wa09]{Wa09} X.-J. Wang, {\it The $k$-Hessian equation,} Lect. Not. Math. {\bf 1977} (2009).

\bibitem[Zha17]{Zha17} D. Zhang, {\it Hessian equations on closed Hermitian manifolds.}  Pacific J. Math. {\bf 291} (2017), no.2, 485--510.

\end{thebibliography}
\end{document}